\documentclass[11pt]{amsart}
\usepackage[utf8]{inputenc}
\usepackage{amsmath, amsthm, amsfonts, amssymb}
\usepackage{tikz-cd}
\usepackage{hyperref}
\usepackage{dsfont} 
\usepackage[capitalise]{cleveref}

\hypersetup{
  colorlinks   = true,
  urlcolor     = blue, 
  linkcolor    = blue,
  citecolor    = blue 
}

\numberwithin{equation}{section}

\theoremstyle{definition}
\newtheorem{theorem}[equation]{Theorem}
\newtheorem{lemma}[equation]{Lemma}
\newtheorem{claim}[equation]{Claim}
\newtheorem{corollary}[equation]{Corollary}
\newtheorem{proposition}[equation]{Proposition}
\newtheorem{definition}[equation]{Definition}
\newtheorem{assumption}[equation]{Assumption}
\newtheorem{question}[equation]{Question}

\newtheorem{remark}[equation]{Remark}

\renewcommand{\div}{{\mathrm{div}}}

\newcommand{\im}{{\mathrm{im}}}
\newcommand{\loc}{{\mathrm{loc}}}
\newcommand{\Ric}{{\mathrm{Ric}}}
\newcommand{\Rm}{{\mathrm{Rm}}}
\newcommand{\Hess}{{\mathrm{Hess}}}
\newcommand{\Tr}{\mathrm{Tr}}

\newcommand{\ALE}{{\mathrm{ALE}}}
\newcommand{\ALF}{{\mathrm{ALF}}}

\newcommand{\norm}[1]{{\lVert #1 \rVert}}

\newcommand{\RR}{{\mathbb{R}}}
\renewcommand{\SS}{{\mathbb{S}}}

\newcommand{\ZZ}{{\mathbb{Z}}}

\newcommand{\cF}{{\mathcal{F}}}

\newcommand{\cH}{{\mathcal{H}}}

\newcommand{\cL}{{\mathcal{L}}}
\newcommand{\cM}{{\mathcal{M}}}

\title{Ricci Flow on ALF manifolds}
\author{Dain Kim \and Tristan Ozuch}
\date{}
\address{MIT, Dept. of Math., 77 Massachusetts Avenue, Cambridge, MA 02139-4307}
\email{dain0327@mit.edu \and ozuch@mit.edu}

\begin{document}
\begin{abstract}
We prove that on ALF $n$-manifolds with $n\ge 4$ the Ricci flow preserves the ALF structure, and develop a weighted Fredholm framework adapted to ALF manifolds. Motivated by Perelman's $\lambda$-functional, we define a renormalized functional $\lambda_{\ALF}$ whose gradient flow is the Ricci flow. It is built from a relative mass with respect to a reference Ricci-flat metric at infinity. This yields a natural notion of variational and linear stability for Ricci-flat ALF $4$-metrics and lets us show that the conformally Kähler, non-hyperkähler examples are dynamically unstable along Ricci flow. We finally relate the sign of $\lambda_{\ALF}$ to positive relative mass statements for ALF metrics.

\end{abstract}

\maketitle

\section*{Introduction}

To understand Ricci flow on 4-manifolds, it is crucial to analyze its singularity models, both at finite time and infinite time. A folklore expectation is that ALF Ricci-flat metrics arise as singularity models in the infinite time. In this work, we develop a framework and analytic tools to study such metrics. Furthermore, we establish stability results, an important property for singularity models, for a broad class of known examples, namely hyperk\"ahler and conformally K\"ahler Ricci-flat metrics.

Complete non-compact 4-manifolds with quadratic curvature decay are often studied via their asymptotic geometry, which groups many known examples into broad classes. The first such classes were ALE (Asymptotically Locally Euclidean), ALF (Asymptotically Locally Flat), ALG, and ALH, which are distinguished by the volume growth. Two additional structures named ALG* and ALH* are found later \cite{Hei12}.
Moreover, if the hyperk\"ahler 4-manifold has faster than quadratic curvature decay, the manifold must belong to one of the classes ALE, ALF, ALG, or ALH \cite{CC21a}. On the other hand, if one only assumes that the curvature is $L^2$-integrable, then the manifold is constrained to be one of ALE, ALF, ALG, ALH, ALG*, or ALH* \cite{SZ24}.
Consequently, ALF manifolds are regarded as the next simplest asymptotic models after ALE, and in particular as the simplest collapsing models.

In contrast to the ALE case, where all known complete Ricci-flat 4-manifolds are hyperk\"ahler, the ALF setting admits non-hyperk\"ahler examples. Two standard examples are the Kerr and Taub-Bolt metrics, with the Schwarzschild metric appearing as a special case of Kerr metrics. Topologically, the Taub-Bolt metric is modeled on the Hopf fibration, while the Kerr metric corresponds to a trivial circle fibration; the latter type of geometry is referred to as asymptotically flat (AF). A na\"ive formulation of the Riemannian black hole uniqueness conjecture \cite{Gib80} suggested that the Kerr metric is the only non-flat AF Ricci-flat 4-manifold. However, this was disproved by the construction of the Chen-Teo metric \cite{CT11}. 

Notably, all these metrics (Kerr, Taub-Bolt, and Chen-Teo) exhibit special geometry: they are conformally K\"ahler.
For 4-dimensional Ricci-flat metrics, being nontrivially conformally K\"ahler is equivalent to being Hermitian but not K\"ahler.
The classification of non-K\"ahler Hermitian ALF Ricci-flat metrics with $L^2$ Riemannian curvature was completed in \cite{Li23}, building on earlier work, including the classification of toric ALF Ricci-flat metrics \cite{BG23,BGL24} and detailed analyses on Chen-Teo metrics \cite{AA24}. It is shown there that the complete list of such metrics consists precisely of the Kerr, Taub-Bolt, and Chen-Teo metrics, together with the Taub-NUT metric endowed with the reversed orientation. Moreover, these metrics are both integrable and infinitesimally rigid \cite{AA25}.

More recently, infinitely many Ricci-flat AF metrics have been constructed that are topologically distinct and not even locally Hermitian, illustrating the rich and diverse landscape of non-hyperk\"ahler Ricci-flat AF 4-manifolds \cite{LS25}. 

\subsection*{Analytic aspects of ALF metrics and Ricci flow}

We begin by developing a self-contained analysis of function spaces on ALF manifolds. Although there has been substantial work in this direction for various metrics \cite{Bar86, Min09, CC21a, CC19, CC21b, CVZ23}, including the general frameworks for fibered boundary metrics and fibered cusp metrics developed by Mazzeo and Melrose, and by Hausel, Hunsicker, and Mazzeo \cite{HHM04}, we choose to present an independent and flexible approach adapted to the ALF setting that does not rely on the $\mathbb{S}^1$-bundle at infinity having closed orbits. This allows us to simplify the arguments and avoid reliance on the full general theory. We also note that a related Fredholm result was obtained by Minerbe \cite{Min09}, who showed that the Laplacian is Fredholm between weighted Sobolev spaces on ALF manifolds. However, his analysis requires assigning different decay rates in the fiber and base directions and is carried out entirely in the Sobolev setting, which does not fit well with our purposes. Nonetheless, certain ideas in our approach were motivated by his work.

We study the Ricci flow on ALF manifolds, a topic that has also attracted interest in the physics community \cite{HW06, HSW07, Tak14}. We also note that Ricci flow has been studied in other asymptotic geometries: the asymptotically Euclidean (AE) case by \cite{OW07}, ALE manifolds \cite{DO20, DO24}, and AH manifolds by \cite{BW12,BW18}. A fundamental first step in studying the Ricci flow on ALF manifolds is to determine whether the ALF structure is preserved under the flow. We show that this is indeed the case; see \Cref{thm:alfp} for a precise statement.
\begin{theorem}
    ALF structure is preserved under the Ricci flow.
\end{theorem}
With this in place, we turn to the central question of this paper: 
\begin{question}
Are non-hyperk\"ahler Ricci-flat ALF manifolds dynamically stable or unstable under the Ricci flow?
\end{question}

\subsection*{A renormalized Perelman $\lambda$-functional on ALF metrics}

The study of dynamical (in)stability of Ricci-flat metrics under Ricci flow has a long history, with foundational contributions by Perelman \cite{Per02} in his resolution of the Poincar\'e conjecture. A key innovation in his work was the introduction of three functionals called $\lambda$, $\mu$, and $\nu$, and the observation that Ricci flow may be interpreted as the gradient flow of these functionals on the space of Riemannian metrics modulo diffeomorphisms. In particular, the $\lambda$-functional is defined by 
\[
    \lambda(g) := \inf_{\lVert \varphi \rVert_{L^2} = 1} \int_M (4|\nabla^g \varphi|_g^2+R_g\varphi^2),
\]
which behaves well on closed manifolds. However, on complete non-compact manifolds, this definition becomes problematic: one can show that $\lambda(g)$, defined using $L^2$-normalized test functions as above yields no useful information about the geometry of $g$.

To address this issue, Haslhofer \cite{Has11} introduced an adapted version of the $\lambda$-functional, where the test functions are required to approach 1 at infinity and differ from 1 by a compactly supported smooth function. This formulation is well suited for ALE manifolds with nonnegative and integrable scalar curvature. Moreover, he compared this adapted functional to ADM mass, which yields a quantitative positive mass theorem. Building on this idea, \cite{DO20} developed the theory in the ALE context. In particular, it is shown that if one subtracts the ADM mass from Haslhofer's adapted $\lambda$-functional, the resulting quantity, which we denote by $\lambda_{\ALE}$, is well-defined in a neighborhood of Ricci-flat metrics in a suitable weighted H\"older spaces. Remarkably, $\lambda_{\ALE}$ remains meaningful even when the scalar curvature is not integrable or the ADM mass is not well-defined, and the Ricci flow continues to behave as a gradient flow of this functional in a weighted $L^2$-sense.

In this work, we extend their approach to the ALF setting. While there is no universally accepted definition of mass for ALF manifolds, though several works exist \cite{Dai04, Min09, KW25}, we introduce a notion of relative mass, defined with respect to a fixed Ricci-flat reference metric. More precisely, we define the relative mass as 
\[
    m(g, g_{\mathrm{RF}}) = \lim_{R \to \infty} \int_{B_R} \div_{g_{\mathrm{RF}}}(\div_{g_{\mathrm{RF}}}(g-g_{\mathrm{RF}}) - \nabla^{g_{\mathrm{RF}}}\Tr(g-g_{\mathrm{RF}})),
\]
where $g_{\mathrm{RF}}$ is a fixed background Ricci-flat metric (defined in a neighborhood of infinity). The relative mass exhibits additive structure, and therefore changing of reference metric modifies the mass only by an additive constant. Consequently, such a change has no geometric effect on the overall theory. We nevertheless believe that the background Ricci-flat metric is the natural choice (see the discussion preceding \Cref{thm:R3S1}).

Using this, we define an analogous functional $\lambda_{\ALF}$, modeled after $\lambda_{\ALE}$, as
\[
    \lambda_{\ALF}(g, g_{\mathrm{RF}}) = \inf_{\varphi - 1 \in C_c^\infty(M)} \int_M (4|\nabla^g \varphi|_g^2 + R_g \varphi^2) - m(g, g_{\mathrm{RF}})
\]
and show that it is well-defined in a neighborhood of ALF metrics that are either close to Ricci-flat metric or have nonnegative and integrable scalar curvature. Moreover, we prove that the Ricci flow is the gradient flow of $\lambda_{\ALF}$ in a weighted $L^2$-sense.

\subsection*{Dynamical instability of Ricci-flat ALF metrics}
As Ricci-flat metrics are critical points of the Ricci flow, it is natural to ask about their stability. In the compact case this question has been studied extensively by Sesum \cite{Ses06}, Haslhofer \cite{Has12}, and Haslhofer and M\"uller \cite{HM14}. Of particular relevance for us is the ALE case, which has been analyzed in \cite{DO20}. Equipped with the $\lambda_{\ALF}$ functional, we now turn to the stability of Ricci-flat ALF metrics. We establish the dynamical instability of non-hyperk\"ahler but conformally K\"ahler ALF metrics:
\begin{theorem}[Dynamical Instability]\label{thm:dyn instability}
    Each of the Kerr, Taub-Bolt, and Chen-Teo metrics is dynamically unstable under Ricci flow.
\end{theorem}
For Ricci-flat ALF metrics, the second variation of $\lambda_{\ALF}$ coincides on traceless and divergence-free $2$-tensors with the the second variation of the Einstein-Hilbert functional, defined by $\mathcal{S}(g) := \int R_g$. As a consequence, our results imply that linear instability directly leads to dynamical instability under the Ricci flow. 
We emphasize that gauge-fixing is crucial, because otherwise
\begin{enumerate}
    \item it is not clear whether the 2-tensors genuinely increase the functional because the second variation behaves poorly without gauge-fixing, and
    \item from a dynamical point of view, one might merely construct nontrivial DeTurck-Ricci flows rather than genuine nontrivial Ricci flows.
\end{enumerate}

In \cite{BO23}, linear instability of the Kerr, Taub-Bolt, and Chen-Teo metrics is demonstrated via explicit variations that decreases $\mathcal{S}$. However, the initial variations used there are not divergence-free, and making them so, which is necessary for fixing the gauge, required the use of the $b$-calculus framework as a black box. Using the analytic tools developed in this work, we show that these variations can be adjusted to be divergence-free without invoking $b$-calculus, thereby providing a more elementary and self-contained proof of the instability of these metrics. Also, our work justifies that these metrics are not only \emph{linearly unstable} but \emph{dynamically unstable} along Ricci flow. 

While related discussions on the dynamical instability of Taub-Bolt metric can be found in \cite{Hug24, Hug25a, Hug25b}, our analysis follows a different approach and is independent of those works.

\subsection*{Positive relative mass, stability and rigidity for ALF metrics}

We conclude by providing a positive mass theorem in terms of $\lambda_{\ALF}$. 
\begin{theorem}\label{thm:R3S1}
    If $(M^{m+1}, g)$ is a spin AF manifold with nonnegative scalar curvature, then 
    \[
        \lambda_{\ALF}(g, \overline{g}) \le 0
    \]
    for a reference metric $\overline{g}$ asymptotic to $\RR^m \times S^1$. The equality holds if and only if $g$ is isometric to $\mathbb{R}^m \times S^1$. 
\end{theorem}

Similar statements hold for other hyperkähler ALF manifolds with flat ends. One such example is a family of metrics conjectured by Page \cite{Pag81} and first rigorously constructed by Hitchin \cite{Hit84} and then by Biquard and Minerbe \cite{BM11} which are now often called the $D_2$ ALF spaces. These metrics decay faster than $r^{-2}$ to the flat metric on $(\mathbb{R}^3\times \mathbb{S}^1)/\mathbb{Z}_2$ with $\mathbb{Z}_2$ acting by $(x,y)\mapsto (-x,-y)$ and therefore have zero relative mass. We show that they are the only such metrics on their topologies with their asymptotics.

\begin{theorem}\label{thm:HP}
    Let $(M^4,g)$ be a metric with nonnegative scalar curvature with $M^4$ the same topology as Hitchin-Page's ALF metrics, or more generally with $M^4$ a spin manifold whose spin structure is compatible with that of the flat $(\mathbb{R}^3\times \mathbb{S}^1)/\mathbb{Z}_2$ at infinity.
    Then one has 
    \[
    m(g) = m(g, \overline{g}) \ge 0 \quad \text{and} \quad \lambda_{\ALF}(g,\overline{g}) = \lambda_{\ALF}(g,g_{HP})\leq 0,
    \]
    where $\overline{g}$ is a metric equal to the flat metric on $(\mathbb{R}^3\times \mathbb{S}^1)/\mathbb{Z}_2$ at infinity, and $g_{HP}$ is one of the metrics of Hitchin-Page. There is equality if and only if $g$ is one of Hitchin-Page's hyperkähler metrics.
\end{theorem}

\begin{remark}
    Theorem \ref{thm:HP} highlights the topological constraints at infinity in \cite{KW25}: within the ALF class they study, the only metric with nonnegative scalar curvature and vanishing mass is $\mathbb{R}^3\times \mathbb{S}^1$. In contrast, in our setting every zero-mass metric is hyperkähler. This parallels the AE positive mass theorem, where equality forces the Euclidean metric, whereas allowing ALE ends yields equality for hyperkähler metrics.
\end{remark}

We interpret Theorems \ref{thm:R3S1} and \ref{thm:HP} as \textit{global} stability of hyperkähler metrics among metrics with nonnegative scalar curvature: they are the absolute maximizers of $\lambda_{\ALF}$. In particular, Ricci flows starting at a metric with nonnegative scalar curvature will always have a $\lambda_{\ALF}$-functional that's lower than the corresponding hyperkähler metrics.
\\

The other hyperkähler ALF metrics have more complicated ALF ends, and to the authors' knowledge, the analysis of spinors in adapted weighted spaces is not available in the literature. 

\begin{question}
    Let $(M,g_{HK})$ be a hyperkähler ALF metrics and $g$ an ALF metric with nonnegative scalar curvature and with the same model metric as $g_{HK}$ at infinity. Then, do we have a positive relative mass theorem:
    $$m(g,g_{HK})\geq 0\qquad\qquad\text{ and }\qquad \qquad\lambda_{\ALF}(g,g_{HK})\leq0$$
    with equality if and only if $g$ is hyperkähler? 
\end{question}

These inequalities would \textit{not} be satisfied with Minerbe's notion of mass since there are hyperkähler ALF metrics with negative mass as discussed in \cite[Section 3]{BM11}. This is one reason why we believe that the notion of relative mass is more relevant to the study of Ricci flow. Another very recent motivation is \cite{KW25}, where the inequality $m(g,g_{HK})\geq 0$ for the multi Taub-NUT hyperkähler ALF metrics is left as an open question on a smaller set of metrics.

Another motivation comes from the rigidity of Ricci-flat metrics on ALF manifolds. A refined version of the Hitchin-Thorpe inequality \cite[(2.14)]{GP80} together with the Chern-Gauss-Bonnet formula for ALF manifolds imposes strong restrictions on the Ricci-flat metrics. For example, in the case of manifolds diffeomorphic to $\RR^4$, the only ALF Ricci-flat metric, up to homothety, is the Taub-NUT metric. This suggests that it is natural to define and compare the mass with respect to such a Ricci-flat reference metric rather than with a model metric, denoted $g_0$ below, with negative scalar curvature if the fibration is nontrivial.

One can also ask the question of the behavior of Ricci flow near the new metrics of \cite{LS25}.
\begin{question}[{Ricci flow analogue of \cite[Question 8.8]{LS25}}] Are the new examples of Ricci-flat metric constructed in \cite{LS25} dynamically or linearly stable along Ricci flow?
\end{question}

Intuitively, one might conjecture from \cite{BO23} that the more selfdual and anti-selfdual harmonic $2$-forms on the manifold, the more unstable it should be. This has however only been verified on a short list of examples.

\subsection*{Organization of the paper} The paper is organized as follows. Section 1 introduces the definition of ALF metrics. Section 2 establishes that the Laplacian is a Fredholm operator in weighted H\"older spaces on ALF manifolds. Section 3 demonstrates that the ALF structure is preserved under the Ricci flow. In Section 4, we define the $\lambda_{\ALF}$ functional for ALF metrics that are close to Ricci-flat metrics in a weighted H\"older norm, and we show that the Ricci flow is the gradient flow of this functional in a weighted $L^2$ sense. Section 5 proves the dynamical instability of non-hyperk\"ahler conformally K\"ahler Ricci-flat ALF metrics. Finally, Section 6 discusses the positive relative mass theorem in the context of $\lambda_{\ALF}$ functional.

\subsection*{Acknowledgements}
The authors are grateful to Alix Deruelle and Rafe Mazzeo for their interests and insightful comments. DK would like to thank Bill Minicozzi for his continous support and inspiring conversations. During this project, DK was partially supported by NSF Grant DMS-2304684 and TO was partially supported by NSF Grant DMS-2405328.

\section{ALF manifolds}
\label{sec:alf}
Let $M^{m+1}$, $m \ge 3$, be a smooth Riemannian manifold such that outside a compact set $K \subset M$, there is a diffeomorphism $\Phi \colon M \setminus K \to (1,\infty) \times E$ where $E$ is the total space of a principal $S^1$-bundle over $\mathbb{S}^{m-1}$. For a given $L>0$, we say that $g_0$ is a \emph{model metric on} $M$ (\emph{with fiber length} $L$) if 
\[
    g_0 = \Phi^\ast(dr^2 + r^2\sigma + \theta^2)
\]
on $M\setminus K$ where $\sigma$ is the standard round metric on $\mathbb{S}^{m-1}$ and $\theta$ is a connection 1-form on $E$ with fiber length $L$. We also define a projection map $\pi \colon M \setminus K \to \RR^m \setminus \overline{B_1(0)} \cong (1,\infty) \times \mathbb{S}^{m-1}$.
Then we call $(M,g)$ \emph{Asymptotically Locally Flat (ALF)} of order $\eta > 0$ if for all $k \ge 0$, 
\[
    r^k|\nabla^{g_0, k}(g-g_0)| = O(r^{-\eta}) \text{ on } M \setminus K 
\]
where $r$ is the projection of $\Phi$ onto the radial interval $(1,\infty)$.

On ALF manifolds $(M,g)$, we define weighted H\"older spaces. 
Fix a smooth positive function $\rho = \rho_{g_0} > 1$ that agrees to $\pi^\ast r_{\RR^m \setminus \overline{B_2(0)}}$.
In particular, note that $\rho$ is invariant under $S^1$-action on $M \setminus K$ and equivalent to $r_g$.
For $\tau \in \RR$ and a tensor $s$ on $M$, we define its $C_\tau^k(M)$ norm as 
\[
    \norm{s}^g_{C_\tau^k(M)} = \sup_M \rho^{\tau}\left( \sum_{i=0}^k \rho^k |\nabla^{g,k}s|_g \right).
\]
For $\alpha \in (0,1)$, we define $C_\tau^{k,\alpha}(M)$ norm as 
\[
    \norm{s}^g_{C_\tau^{k,\alpha}(M)} = \sup_{M} \rho^\tau \left( \sum_{i=0}^k \rho^k |\nabla^{g,k}s|_g + \rho^{k+\alpha} [\nabla^{g,k}s]_{C^{0,\alpha}} \right),
\]
where $[\phi]_{C^{0,\alpha}}$ is defined on $B_{\rho(x)/2}(x)$ for each $x \in M$. 

\section{Fredholmness of Laplacians on ALF manifolds}
\label{sec:fl}

We now proceed to the proof that the Laplacian is Fredholm on ALF manifolds. Since Fredholmness is preserved under compact perturbations, it suffices to establish the Fredholmness for ALF manifolds whose metric agrees with the model metric at infinity. In other words, we first treat the case of ALF manifolds that admit a principal $S^1$-bundle structure outside a compact set.

\begin{assumption}
\begin{enumerate}
    \item $m \ge 3$, $\eta > 0$, and $\alpha \in (0,1)$ are given.
    \item $(M, g)$ is an ALF manifold of order $\eta$ with a model metric $g_0$ whose fiber length is $L>0$.
    \item Fix $o \in M$. $R_0 > 100L$ is large enough so that $K \subset B_{R_0}(o)$.
    \item On $M \setminus B_{R_0}(0)$, the metric $g$ agrees with the model metric $g_0$.
\end{enumerate}
\label{ass:fl}
\end{assumption}

Unless otherwise stated, the constant $C=C(M,\eta)>0$ may increase from line to line in the proofs.
Throughout this section, we assume the followings.
When it is clear from the context, we drop the center of the ball and simply write $B_{R_0} = B_{R_0}(o)$.

\subsection{Exceptional values of $\Delta_g$}
To analyze the growth of harmonic functions on $M$, we decompose functions into two parts: a function that is independent of the fiber and a function that integrates to 0 over each fiber but can take different values within each fiber. This decomposition was motivated by \cite{Min09}.

Note that $M \setminus K$ naturally inherits a $S^1$-principal bundle structure. For $u \colon M \setminus K \to \RR$, define $u^\pi, u^\perp \colon M\setminus K \to \RR$ by 
\begin{align*}
u^\pi(x) &= \int_{\iota \in S^1} u(\iota \cdot x),\\
u^\perp(x) &= \left(u- \frac{1}{L} u^\pi\right)(x),
\end{align*}
so that $u^\perp$ has average zero on each fiber.
We note two facts regarding to above decomposition.
\begin{itemize}
\item If $u \in C_\tau^{2,\alpha}(M)$, then both $u^\pi, u^\perp \in C_\tau^{2,\alpha}(M \setminus K)$.

\item If $u$ is harmonic, then 
\[
    \Delta_g u^\pi(x) = \int_{\iota \in S^1} (\Delta_g u)(\iota \cdot x) = 0
\]
for $x \in M \setminus K$, and therefore $\Delta_g u^\perp = \Delta_g (u-u^\pi) = 0$ on $M \setminus K$. So both $u^\pi$ and $u^\perp$ are harmonic on $M \setminus K$.
\end{itemize}

\begin{proposition}
Let $(M,g)$ be an ALF manifold as in \Cref{ass:fl}.
For $\tau \in \RR$, if $u \in C_\tau^{2,\alpha}(M \setminus K)$ is harmonic, then $u^\perp$ decays faster than any polynomials, i.e. $u^\perp = o(\rho^{-\kappa})$ for any $\kappa > 0$. 
\label{prop:gr}
\end{proposition}

The main technical step of the proof is the following lemma.

\begin{lemma}
Given $0 < \nu < 1$, there exists $\rho_0 = \rho_0(M, \nu)>0$ such that for any $x \in M \setminus K$ with $\rho(x) \ge \rho_0$, on a cover $\pi(B_{2\rho(x)^\nu}(x)) \times \RR$ of $\pi^{-1}(\pi(B_{2\rho(x)^\nu}(x)))$ and $u \in C_\tau^{2,\alpha}(M)$ a harmonic function,
\[
    \left| u^\perp(x) \right| \le \frac{2}{\omega_{m+1} (\rho(x))^{\nu(m+1)}} \left| \int_{B_{\rho(x)^\nu}^{\text{cover}}(x)} u^\perp \right|,
\]
where $B_{\rho(x)^\nu}^{\text{cover}} \subset \pi(B_{2\rho(x)^\nu}(x)) \times \RR$ is a ball of radius $\rho(x)^\nu$ in $\pi(B_{2\rho(x)^\nu}(x)) \times \RR$ centered at $(x,0)$ and $\omega_{m+1}$ is the volume of unit ball $B_1^{m+1}(0) \subset \RR^{m+1}$.
\label{lem:mvp}
\end{lemma}
\begin{proof}
For the sake of contradiction, suppose not.
Then we have a sequence $x_k \in M \setminus K$ with $\rho(x_k) \to \infty$ that violate the desired inequality. By unique continuation, nonzero  harmonic functions cannot vanish locally, so we may assume that $u^\perp(x_k) \ne 0$ by perturbations if needed.
Define $v_k \colon B_1^m(0) \times \RR \to \RR$ as 
\[
    v_k(y_0, y_1) = \dfrac{u^\perp(x_k+\rho(x_k)^\nu y_0, \rho(x_k)^\nu y_1)}{\sup_{B_{\frac{3}{2}\rho(x_k)^\nu}} \left(u^\perp, \rho(x_k)^\nu |\nabla u^\perp|, \rho(x_k)^{2\nu}|\nabla^2 u^\perp|\right)}
\]
so that $\Delta_{g_k}v_k = 0$ where $g_k$ is a metric on $B_2^m(0) \times \RR$ given by dilating $\Phi^\ast g$ around $x_k$ by a factor of $\rho(x)^{-\nu}$ and translating $x_k$ to the origin. Note that this procedure is not canonical as the metric depends on the choice of coordinates to start with. However, as $0<\nu < 1$, the metric $g_k$ converges to the standard Euclidean metric, and therefore $\Delta_{g_k}$ converges to the standard Laplacian $\Delta_{\text{euc}}$ on $\RR^{m+1}$.

Note that as $\sup_{B_{3/2}}(v_k, |\nabla v_k|, |\nabla^2 v_k|) = 1$ and $v_k \in C^{2,\alpha}(B_{2}(0) \times \RR)$, there exists a subsequence that converges to a nontrivial $v_\infty \in C^{2,\alpha/2}(B_{\frac{3}{2}}(0) \times (-2,2))$.
Since $\Delta_{g_k}$ converges to $\Delta_{\text{euc}}$ on $B_{\frac{3}{2}}(0) \times (-2,2)$, $\Delta_\text{euc} v_\infty = 0$. But then by mean value property of harmonic functions on Euclidean space, 
\[
    v_\infty(0) = \dfrac{1}{\omega_{m+1}}\int_{B_1 \subset B^m_{\frac{3}{2}}(0) \times (-2,2)} v_\infty,
\]
which contradicts the facts that by construction for all $k$,
\[
    \left| v_k(0) \right| > \frac{2}{\omega_{m+1}} \left| \int_{B_1 \subset B^m_{2}(0) \times \RR} v_k \right|,
\]
and that $v_\infty$ is the limit of $v_k$. Here, if $v_\infty(0) = 0$ then we do not get the contradiction immediately, but we may slightly perturb the origin to get a contradiction using the fact that $v_\infty$ is a nontrivial harmonic function and $v_k$ uniformly converges to $v_\infty$ in $C^{2,\alpha/2}$.
\end{proof}

\begin{proof}[Proof of \Cref{prop:gr}]
We prove the statement when $\tau > 0$ first.
We proceed by induction. 
First, assume that $0<\kappa<1$ and take $\nu = \frac{1+\kappa}{2}$ in \Cref{lem:mvp} to get $\rho_0>0$.
Note that as $\tau > 0$, $\left| \sup u^\perp \right| < +\infty$.

We want an estimate on $u^\perp(x)$ for $x \in M \setminus K$ with $\rho(x) \ge \rho_0$.
Note that $u^\perp$ integrates to 0 on each fiber, and therefore, if we let $p \colon B_{\rho(x)^\nu}^{\text{cover}} \to \pi (B_{2\rho(x)^\nu}(x))$ be the projection onto the second argument, then 
\begin{equation}
    \left| \int_{p^{-1}(s)} u^\perp \right| \le L|\sup u^\perp|
    \label{eqn:ce}
\end{equation}
for each $s \in \pi(B_{2\rho(x)^\nu}(x))$.
Hence, by \Cref{lem:mvp}, 
\begin{align*}
    \left| u^\perp(x) \right|
    &\le \frac{2}{\omega_{m+1}(\rho(x))^{\nu(m+1)}} \left| \int_{B_{\rho(x)^\nu}^{\text{cover}}} u^\perp \right|\\
    &\le \dfrac{2}{\omega_{m+1} (\rho(x))^{\nu(m+1)} }\cH^m(B_{2\rho(x)^\nu}(x))L\left|\sup u^\perp\right|\\
    &= O(\rho(x)^{-\nu}) = O(\rho(x)^{-(1+\kappa)/2}).
\end{align*}
In particular, as $(1+\kappa)/2 > \kappa$, we conclude that $u^\perp = o(\rho^{-\kappa})$.

Now assume that for an integer $n$, we know $u^\perp = o(\rho^{-\kappa_n})$ for any positive integer $n$ and $\dfrac{n-1}{2} < \kappa_n < \dfrac{n}{2}$.
For $\dfrac{n}{2} < \kappa_{n+1} < \dfrac{n+1}{2}$, the same argument with one modification that instead of \eqref{eqn:ce} we use 
\[
    \left| \int_{p^{-1}(s)} u^\perp \right| \le L C_{\kappa_{n+1}-\frac{1}{2}}\rho(x)^{\frac{1}{2}-\kappa_{n+1}}
\]
for a constant $C_{\kappa_{n+1}-\frac{1}{2}}>0$, which holds by the induction hypothesis, with $\nu = \frac{1}{2}$ give $u^\perp = O(\rho^{-\kappa_{n+1}})$. 

The proof when $\tau <0$ is similar. Instead of arguing with $|\sup u^\perp| < +\infty$, we use $|u(x)| \le C\norm {u}_{C_\tau^{2,\alpha}(M)} \rho(x)^{-\tau}$ and by the same induction argument we decrease the exponent by $\frac{1}{2}$ repeatedly.
\end{proof}

The arguments so far ultimately show that the polynomial growth rate of harmonic functions is controlled by $u^\pi$, which reduces to the understanding of harmonic functions in the Asymptotically Locally Euclidean (ALE) setting, particularly in the model case of Euclidean space.
Since the literature on this topic is readily available, we list some facts we intend to use below without proof.

\begin{proposition}
    Let $\cH_k(\RR^m)$ denote the space of homogeneous harmonic polynomials of degree $k$ on $\RR^m$, and define 
    \[
        \cH_{\le k}(\RR^m) := \bigoplus_{1 \le l \le k} \cH_l(\RR^m).
    \]
    Then every harmonic function on $\RR^m$ of at most polynomial growth is a harmonic polynomial, and we have 
    \[
        \dim \cH_k(\RR^m) = \binom{m+k-1}{m-1} - \binom{m+k-3}{m-1}.
    \]
    In particular, both $\cH_k(\RR^m)$ and $\cH_{\le k}(\RR^m)$ are finite-dimensional.
    \label{prop:ale1}
\end{proposition}

\begin{proposition}
    Let $\cH_k(\RR^m \setminus \{0\})$ be the space of homogeneous harmonic functions of degree $k$ on $\RR^m \setminus \{0\}$. Then
    \[
        \cH_k(\RR^m \setminus \{0\}) =
        \begin{cases}
            \cH_k(\RR^m) &\text{ for } k \ge 0,\\
            \cH_{k+m-2} (\mathbb{S}^{m-1}) &\text{ for } k \le 2-m,\\
            \{0\} &\text{ otherwise},
        \end{cases}
    \]
    where $\cH_k(\SS^{m-1})$ is the eigenspace of the spherical Laplacian on $\SS^{m-1}$ with eigenvalue $k(k+m-2)$.
    \label{prop:ale2}
\end{proposition}

Accordingly, we define the set of exceptional values $\Lambda = \ZZ \setminus (0, m-2)$ to be the set of all possible polynomial decays of harmonic functions on $\RR^m \setminus \{0\}$. Then by our previous arguments, the exceptional values of ALF manifolds are also $\Lambda = \ZZ \setminus (0, m-2)$.

\begin{proposition}
Let $\tau_1 < \tau_2$ be two real numbers that are not included in $\Lambda$. 
If $u \in C_{\tau_1}^{2,\alpha}(\RR^m \setminus \overline{B_1(0)})$ and $\Delta u \in C_{\tau_2}^{0,\alpha}(\RR^m \setminus \overline{B_1(0)})$, then for each $k \in (\tau_1, \tau_2) \cap \Lambda$, there exists $h_k \in \cH_k(\RR^m \setminus \{0\})$ such that 
\[
    u - \sum_{k \in (\tau_1,  \tau_2) \cap \Lambda} h_k \in C_{\tau_2}^{2,\alpha}(\RR^m \setminus \overline{B_1(0)}).
\]
\label{prop:ale3}
\end{proposition}

\subsection{Main Estimates}

\begin{proposition}
    Let $(M,g)$ be an ALF manifold as in \Cref{ass:fl}. For $\tau \in \RR$, there exists a constant $C=C(M,\tau)>0$ such that 
    \[
        \norm{u}_{C_\tau^{2,\alpha}(M)} \le C(\norm{\Delta_gu}_{C_{\tau+2}^{0,\alpha}(M)}+\norm{u}_{C_\tau^{0}(M)})
    \]
    for all $u \in C_\tau^{2,\alpha}(M)$.
    \label{prop:ape}
\end{proposition}

\begin{proof}[Proof of \Cref{prop:ape}]
    Define annuli $A_{a,b} := B^M_{2^bR_0}\setminus \overline{ B^M_{2^aR_0} }$ for $b>a\ge 1$. 
    Let $k \ge 3$ and $x \in A_{k-1,k+1}$. Then $\pi^{-1}(\pi(B_{2^{k-2}R_0}(x))) \subset M$ is a $S^1$-bundle over the simply connected domain $\pi(B_{2^{k-2}R_0}(x))$, so in particular it is a topologically trivial $S^1$-bundle.
    By applying interior elliptic estimate (c.f. \cite[Theorem 6.2]{GT}) on a cover $\pi(B_{2^{k-2}R_0}(x)) \times \RR$ of the bundle, we deduce 
    \begin{align}
        &\sup_{B^M_{2^{k-3}R_0}(x)} \left( |u| + (2^kR_0) |\nabla u| + (2^kR_0)^2 |\nabla^2 u| \right) + (2^kR_0)^{2+\alpha} \sup_{y \in B^M_{2^{k-3}R_0}(x)} \dfrac{|\nabla^2u(y)-\nabla^2u(z)|}{|y-z|^\alpha} \nonumber\\
        &\;\le C \left( \sup_{B_{2^{k-2}R_0}(x)} |u| + (2^kR_0)^2 \sup_{B_{2^{k-2}R_0}(x)} |\Delta_g u| + (2^kR_0)^{2+\alpha} \sup_{y \in B^M_{2^{k-2}R_0}(x)} \dfrac{|\Delta_g u(y)- \Delta_g u(z)|}{|y-z|^\alpha} \right). \label{eqn:ape1}
    \end{align}
    By Vitali covering lemma, we can find a finite cover $\{B^M_{2^{k-4}R_0}(x_i) : x_i \in A_{k-1,k+1}\}$ of $A_{k-1,k+1}$ such that each point is covered at most a certain times independent of $k$ and $R_0$, so by adding up the estimate \eqref{eqn:ape1} for open sets in the cover, we obtain 
    \begin{align}
        &\sup_{A_{k-1,k+1}} \left( |u| + (2^kR_0) |\nabla u| + (2^kR_0)^2 |\nabla^2 u| \right) + (2^kR_0)^{2+\alpha} \sup_{y \in A_{k-1,k+1}} \dfrac{|\nabla^2 u(y)-\nabla^2 u(z)|}{|y-z|^\alpha} \nonumber\\
        &\le C \left( \sup_{y \in A_{k-2,k+2}} |u| + (2^kR_0)^2 \sup_{A_{k-2,k+2}} |\Delta_g u| + (2^kR_0)^{2+\alpha} \sup_{y \in A_{k-2,k+2}} \dfrac{|\Delta_g u(y)-\Delta_g u(z)|}{|y-z|^\alpha} \right), \label{eqn:ape2}
    \end{align}
    where $C$ is independent of $k$.
    Rescaling \eqref{eqn:ape2} by $(2^kR_0)^\tau$ and adding up for $k \ge 3$ yields  
    \begin{align}
        &\sup_{M \setminus B^M_{5R_0}} \rho^\tau \left( |u| + \rho |\nabla u| + \rho^2 |\nabla^2 u| + \rho^{2+\alpha} [\nabla^2 u]_{C^{0,\alpha}} \right) \nonumber\\
        &\quad\le C \sup_{M \setminus B^M_{2R_0}} \rho^\tau \left( |u| + \rho^2 |\Delta_g u| + \rho^{2+\alpha}[\Delta_g u]_{C^{0,\alpha}} \right) \label{eqn:ape3}
    \end{align}
    Take a cutoff function $\chi \in C_c^\infty(B^M_{8R_0}) \cap C^\infty(M)$ such that $\chi \equiv 1$ on $B^M_{6R_0}$.
    By interior estimates again on a covering of $B_{8R_0}^M$ and patching them together, we have 
    \begin{align}
        \norm{u}_{C^{2,\alpha}(B^M_{6R_0})} 
        &\le \norm{\chi u}_{C^{2,\alpha}(B^M_{8R_0})}\nonumber\\
        &\le C (\norm{\Delta_g(\chi u)}_{C^{\alpha}(B^M_{8R_0})} + \norm{\chi u}_{C_\tau^0(B_{8R_0}^M)} )\nonumber\\
        &\le C(\norm{\Delta_gu}_{C_\tau^{\alpha}(M)} + \norm{u}_{C_\tau^0(M)}). 
        \label{eqn:ape4}
    \end{align}
    Then \eqref{eqn:ape3} and \eqref{eqn:ape4} give the desired estimate.
\end{proof}

A similar proof technique can be used to bootstrap the regularity of $u$. For $\tau \in \RR$, define $L_\tau^1(M)$ norm as 
\[
    \norm{u}_{L_{\tau}^1(M)} = \int_M \rho^{\tau-m} |u|.
\]
\begin{proposition}
    For $\tau \in \RR$, if $u \in L_{\tau}^1(M)$ and $\Delta_g u \in C_{\tau+2}^{0,\alpha}(M)$, then $u \in C_\tau^{2,\alpha}(M)$.
    \label{prop:bs}
\end{proposition}
\begin{proof}
    Since elliptic regularity gives $u \in C^{2,\alpha}(M)$, we only need to prove that $\norm{u}_{C_\tau^{2,\alpha}(M)} < \infty$.
    We repeat our previous proof of \Cref{prop:ape} but instead for an interpolation inequality. 
    So substitute \eqref{eqn:ape1} by an interpolation inequality on a cover to get 
    \begin{align*}
        &\sup_{B_{2^{k-3}R_0}^M(x)} \left( |u| + (2^kR_0)|\nabla u| + (2^kR_0)^2 |\nabla^2 u| + (2^kR_0)^{2+\alpha} \dfrac{|\nabla^2 u(y) - \nabla^2 u(z)|}{|y-z|^\alpha} \right)\\
        &\le C \left( (2^kR_0)^2 \sup_{B_{2^{k-2}R_0}^M(x)} \dfrac{|\Delta u(y) - \Delta u(z)|}{|y-z|^\alpha} + (2^kR_0)^{-(m+1)} \int_{B_{2^{k-2}R_0}^M(x)} |u| \right).
    \end{align*}
    Rescaling it by $(2^kR_0)^\tau$ and noting that in fiber direction $u$ is repeated around $\frac{L}{\rho}$ times, we get 
    \[
        \norm{u}_{C_\tau^{2,\alpha}(A_{k-1,k+1})} \le C(\norm{\Delta_g u}_{C_{\tau+2}^{0,\alpha}(A_{k-2,k+2})} + \norm{u}_{L_{\tau}^1(A_{k-2,k+2})})
    \]
    So adding this for $k \ge 3$ and an interpolation inequality on a compact set yields that $\norm{u}_{C_\tau^{2,\alpha}(M)}$ is finite.
\end{proof}

\begin{proposition}
    Let $(M,g)$ be an ALF manifold as in \Cref{ass:fl}. For $\tau \in \RR \setminus \Lambda$, there exist a constant $C=C(M,\tau)>0$ and a compact set $D \subset M$ such that 
    \[
        \norm{u}_{C_\tau^{2,\alpha}(M)} \le C(\norm{\Delta_g u}_{C_{\tau+2}^{0,\alpha}(M)} + \norm{u}_{L^1(D)})
    \]
    for all $u \in C_\tau^{2,\alpha}(M)$.
    \label{prop:me}
\end{proposition}

\begin{lemma}
    For $\tau \in \RR$, there exists $C = C(M,\tau)>0$ satisfying the following: If $u \in C_\tau^{2,\alpha}(M)$ integrates to 0 on each $S^1$-fiber in the region of $M \setminus K$, then 
    \[
        \norm{u}_{C_\tau^0(M)} \le C(\norm{\Delta_g u}_{C_{\tau+2}^{0,\alpha}(M)} + \norm{u}_{L^1(K)}).
    \] 
    \label{lem:de} 
\end{lemma}

\begin{proof}
For the sake of contradiction, suppose there exists a sequence $\{u_k\}_k \subset C_\tau^{2,\alpha}(M)$ which integrates to 0 on each fiber,
\[
    \norm{u_k}_{C_\tau^0(M)} = 1 \text{ and } \norm{\Delta_g u_k}_{C_{\tau+2}^{0,\alpha}} + \norm{u_k}_{L^1(K)} \to 0.
\]
Since $\norm{u_k}_{C_\tau^0(M)}=1$, there exists $p_k \in M$ such that $1-\frac 1k < \rho(p_k)^\tau |u_k(p_k)| \le 1$. 
By flipping the sign of $u_k$ if needed, we may assume that $1-\frac 1k < \rho(p_k)^\tau u_k(p_k) \le 1$.
We divide into two cases depending on whether $\rho(p_k)$ is bounded.

We first assume that $\rho(p_k)$ is bounded.
Then there exists a subsequence of $p_k$ that converges to $p \in M$. Note that as $\norm{u_k}_{C_\tau^0(M)}$ and $\norm{\Delta_g u_k}_{C_{\tau+2}^{0,\alpha}(M)}$ are bounded, by \Cref{prop:ape}, $\norm{u_k}_{C_\tau^{2,\alpha}(M)}$ is bounded as well. Therefore, there exists a subsequence of $u_k$ that locally converges to $u_\infty \in C_{\tau'}^{2,\alpha/2}(M)$ for some $\tau' < \tau$.
This limit $u_\infty$ satisfies 
\[
    u_\infty(p) = \frac{1}{\rho(p)^\tau}, \quad \Delta_g u_\infty = 0, \text{ and } \norm{u_\infty}_{L^1(K)} = 0.
\] 
That is, $u_\infty$ is a nontrivial harmonic function vanishing on a compact set, giving a contradiction.

We now assume that $\rho(p_k)$ is unbounded.
We use a trick similar to the one in the proof of \Cref{lem:mvp}.
Here, we define $v_k \colon B_{\rho(p_k)^\nu}^m(0) \times S^1 \to \RR$ as 
\[
    v_k(y_0, y_1) = \rho(p_k)^\tau u_k(p_k+y_0, y_1)
\]
for coordinates around $p_k$ on which $\Delta_g$ converges to $\Delta_{\text{euc}}$ as $k \to \infty$.
Since $\norm{u_k}_{C_\tau^0(M)}$ and $\norm{\Delta_g u_k}_{C_{\tau+2}^{0,\alpha}(M)}$ are bounded, $\norm{v_k}_{C^0(B_{\rho(p_k)^\nu}^m(0)\times S^1)}$ and $\norm{\Delta_g v_k}_{C^{0,\alpha}(B_{\rho(p_k)^\nu}^m(0) \times S^1)}$ are bounded independent of $\rho(p_k)$ and $k$.
Note that by \Cref{prop:ape}, $\norm{u_k}_{C_\tau^{2,\alpha}(M)}$ are bounded independent of $k$, and therefore, $\norm{v_k}_{C^{2,\alpha}(B_{\rho(p_k)^\nu}^m \times S^1)}$ are bounded. So there exists a subsequence of $v_k$ that converges locally to $v_\infty \in C^{2,\alpha/2}(\RR^m \times S^1)$ such that
\[
    v_\infty(0) = 1, \quad \Delta_g v_\infty = 0,
\]
and $v_\infty$ integrates to 0 on each fiber. 
Viewing $v_\infty$ as a function on $\RR^{m+1}$, a cover of $\RR^m \times S^1$, we know that $v_\infty$ is a harmonic polynomial. Also, periodicity in fiber direction implies that this harmonic polynomial is invariant under $S^1$-action. But then $v_\infty$ does not integrate to 0 on $\{0\} \times S^1 \subset \RR^m \times S^1$, giving a contradiction.
\end{proof}

\begin{proof}[Proof of \Cref{prop:me}]
    By \Cref{prop:ape}, it is enough to prove that for any $\epsilon > 0$, there holds 
    \[
        \norm{u}_{C_\tau^0(M)} \le \epsilon \norm{u}_{C_\tau^{2,\alpha}(M)} + C\norm{\Delta_g u}_{C_{\tau+2}^{0,\alpha}(M)} + C(\epsilon) \norm{u}_{L^1(D)}.
    \]
    Fix $\chi$ a smooth cutoff function such that $\chi \equiv 1$ on $B_{R_0}$, $\chi \equiv 0$ outside $B_{2R_0}$, and taking values in $[0,1]$.
    We decompose $u$ into 3 parts:
    \[
        u = \chi u + ((1-\chi)u)^\pi + ((1-\chi)u)^\perp.
    \]
    Denote by $u_1 = \chi u$, $u_2 = ((1-\chi)u)^\pi$, and $u_3 = ((1-\chi)u)^\perp$.

    Since $u_1$ is compactly supported in $B_{2R_0}$, by an interpolation inequality, 
    \begin{align}
        \norm{u_1}_{C_\tau^0(M)}
        &\le C \norm{u}_{C^0(\overline{B_{2R_0}})} \nonumber\\
        &\le \epsilon \norm{u}_{C^{2,\alpha}(\overline{B_{2R_0}})} + C(\epsilon) \norm{u}_{L^1(\overline{B_{2R_0}})} \label{eqn:me1}
    \end{align}

    Note that $\pi_\ast u_2$ is defined on $\RR^m \setminus \overline{B_1(0)}$. Since $\Delta_{\text{euc}}$ is Fredholm on $\RR^m$ for $\tau \in \RR \setminus \Lambda$, we have 
    \[
        \norm{\pi_\ast u_2}_{C_\tau^{2,\alpha}(\RR^m)} \le C(\norm{\Delta \pi_\ast u_2}_{C_\tau^{0,\alpha}(\RR^m)} + \norm{\pi_\ast u_2}_{L^1(K)}).
    \]
    Therefore, up to a constant, there holds
    \begin{equation}
        \norm{u_2}_{C_\tau^{2,\alpha}(M)} \le C(\norm{\Delta_g u_2}_{C_\tau^{0,\alpha}(M)} + \norm{u_2}_{L^1(K)}).
        \label{eqn:me2}
    \end{equation}

    For $u_3$, we appeal to \Cref{lem:de} to get 
    \begin{equation}
        \norm{u_3}_{C_\tau^0(M)} \le C(\norm{\Delta_g u_3}_{C_{\tau+2}^{0,\alpha}(M)} + \norm{u_3}_{L^1(K)}).
        \label{eqn:me3}
    \end{equation}

    Now we would like to add up \eqref{eqn:me1}, \eqref{eqn:me2}, and \eqref{eqn:me3}. To get the desired estimate with $D = \overline{B_{2R_0}}$, it remains to prove
    \begin{equation}
        \begin{cases}
            \norm{\Delta_g u_2}_{C_\tau^{0,\alpha}(M)} + \norm{\Delta_g u_3}_{C_\tau^{0,\alpha}(M)} \le C\norm{\Delta_g u}_{C_\tau^{0,\alpha}(M)}\\
            \norm{u_2}_{L^1(K)} + \norm{u_3}_{L^1(K)} \le C\norm{u}_{L^1(K)}.
        \end{cases}
        \label{eqn:men}
    \end{equation}

    \begin{claim}
    Let $a \in \RR$ and $b \colon S^1 \to \RR$ a continuous function with $\int_{S^1} b = 0$. Then 
    \[
        |a| + \sup |b| \le 10 \sup |a+b|.
    \]
    \label{cl:me1}
    \end{claim}
    \begin{proof}[Proof of \Cref{cl:me1}]
        If $a = 0$, then the inequality is trivial.
        Hence, we prove for $a \ne 0$ only, and by rescaling we may assume that $a = 1$.

        Suppose that $\sup |b| \le 9$. Since $\int b = 0$ and $b$ is continuous, $0 \in b(S^1)$. That is, $1 \le \sup |1+b|$. On the other hand, the left hand side is $1 + \sup|b| \le 1 + 9 = 10$, so the inequality holds.

        Suppose now that $\sup |b| > 9$. Denote by $b_0 =\sup |b|$. Then
        \[
            1+\sup|b| = 1+b_0 \le 10(b_0-1) \le 10 \sup|1+b|.
        \]
        Hence, the inequality is proven.
    \end{proof}

    \begin{claim}
    Let $a \in \RR$ and $b \colon S^1 \to \RR$ a continuous function with $\int_{S^1} b = 0$. Then 
    \[
        \int_{S^1} (|a|+|b|) \le 10 \int |a+b|.
    \]
    \label{cl:me2}
    \end{claim}
    \begin{proof}
        If $a = 0$, then the inequality is trivial.
        Hence, we prove for $a \ne 0$ only, and by rescaling we may assume that $a = 1$.

        By triangle inequality, note that 
        \[
            \int |1+b| \ge \left| \int (1+b) \right| = L.
        \]
        Therefore, 
        \begin{equation}
            \int_{S^1} |a| \le \int |1+b|.
            \label{eqn:mecl21}
        \end{equation}

        Define $I = \{t \in S^1 : b(t) \in (-5,5)\}$. Then 
        \begin{equation}
            \int_{I} |b| \le 5L \le 5\int |1+b|.
            \label{eqn:mecl22}
        \end{equation}
        Finally, note that for a real number $x$ with $|x| \ge 5$, $|x| \le 4|1+x|$. Therefore, 
        \begin{equation}
            \int_{S^1 \setminus I} |b| \le 4\int |1+b|.
            \label{eqn:mecl23}
        \end{equation}
        Adding up \eqref{eqn:mecl21}, \eqref{eqn:mecl22}, and \eqref{eqn:mecl23} gives the desired inequality.
    \end{proof} 
    By applying \Cref{cl:me1} and \Cref{cl:me2} on each fiber, we get \eqref{eqn:men}.
    Hence the desired estimate is proven.
\end{proof}

\subsection{Fredholmness of $\Delta_g$}

\begin{theorem}
    Let $(M,g)$ be an ALF manifold as in \Cref{ass:fl}.
    For $\tau \in \RR \setminus \Lambda$, the Laplacian $\Delta_g \colon C_\tau^{2,\alpha}(M) \to C_\tau^{0,\alpha}(M)$ is Fredholm.
    \label{prop:fl}
\end{theorem}

\begin{proof}[Proof of \Cref{prop:fl}]
    We first show that $\Delta_g \colon C_\tau^{2,\alpha}(M) \to C_{\tau+2}^{0,\alpha}(M)$ has finite dimensional kernel.

    If $\tau > 0$, then the kernel is trivial by maximum principle. So we assume that $\tau < 0$.
    Let $u \in C_\tau^{2,\alpha}(M)$.
    By \Cref{prop:ale3}, for any small $\epsilon >0$ there exists a harmonic $h \in \cH_{\le -\tau}(\RR^m \setminus \overline{B_1(0)})$ such that 
    \[
        \pi_\ast u^\pi - h \in C_{\epsilon}^{2,\alpha}(\RR^m \setminus \overline{B_1(0)}).
    \]
    We claim that such $h$ is unique. Suppose that $\pi_\ast u^\pi - h_1, \pi_\ast u^\pi - h_2 \in C_\epsilon(\RR^m \setminus \overline{B_1(0)})$.
    Then $h_1 - h_2 \in C_\epsilon(\RR^m \setminus \overline{B_1(0)})$ is a harmonic polynomial, which therefore cannot decay at infinity unless is identically zero.

    We define $\Phi \colon \ker \Delta_g \to \cH_{\le -\tau}(\RR^m \setminus \{0\})$ by $u \mapsto h$. 
    We prove that the map is injective.
    Suppose $\Phi(u) = 0$. That is, $\pi_\ast u^\pi \in C_\epsilon(\RR^m \setminus \overline{B_1(0)})$.
    Also, by \Cref{prop:gr}, $u^\perp$ decays fast.
    Hence, $u$ is harmonic decaying fast at infinity, contradicting the maximum principle unless $u = 0$.
    Therefore, $\Phi$ is injective. Since $\cH_{\le -\tau}(\RR^m \setminus \overline{B_1(0)})$ is finite dimensional, we conclude that the kernel is finite dimensional.

    Next, we prove that 
    \begin{equation}
    \ker(\Delta_g \colon C_\tau^{2,\alpha}(M) \to C_{\tau+2}^{0,\alpha}(M)) = \im(\Delta_g \colon C_{-\tau+m-2}^{2,\alpha}(M) \to C_{-\tau+m}^{0,\alpha}(M))^\perp,
    \label{eqn:kc}
    \end{equation}
    where the spaces are regarded as subspaces of $(C^{0,\alpha}_{-\tau+m}(M))^\ast$.

    We first show $\ker \subset \im^\perp$. For harmonic $u \in C_\tau^{2,\alpha}(M)$, it is enough to show that $\int_M u \Delta_g v  =0$ for all $v \in C_{-\tau+m}^{2,\alpha}(M)$.
    Indeed, by integration by parts, 
    \[
        \int_M u \Delta_g v = - \int_M \langle \nabla^g u, \nabla^g v \rangle = \int_M (\Delta_g u)v = 0,
    \]
    where we used the fact that $u\nabla^g v$ and $v\nabla^g u$ decay fast enough so that the boundary terms are negligible.

    We next show that $\im^\perp \subset \ker$. For $u \in (C_{-\tau+m}^{0,\alpha}(M))^\ast$ that annihilates the image of $\Delta_g$ in $C_{-\tau+m}^{0,\alpha+2}(M)$,
    in particular $u$ annihilates $\Delta_g v$ for all compactly supported smooth functions $v$ on $M$, and therefore $u$ is weakly harmonic on $M$. So by elliptic regularity, $u \in C^\infty(M) \subset C^{2,\alpha}(M)$ and $u$ is harmonic.
    Moreover, $u \in (C_{-\tau+m}^{0,\alpha}(M))^\ast$ and $\rho^{\tau-m} \in C_{-\tau+m}^{0,\alpha}(M)$, so in particular $u \in L_\tau^1(M)$. By \Cref{prop:bs}, $u \in C_\tau^{2,\alpha}(M)$.
\end{proof}

\begin{corollary}
The Laplacian $\Delta_g \colon C_\tau^{2,\alpha}(M) \to C_{\tau+2}^{0,\alpha}(M)$ is Fredholm of index 0 for $\tau \in (0, m-2)$.
\end{corollary}
\begin{proof}
    Since $\Lambda = \ZZ \setminus (0, m-2)$ and there is no harmonic functions decaying at infinity by maximum principle, we know that $\ker(\Delta_g \colon C_\tau^{2,\alpha}(M) \to C_{\tau+2}^{0,\alpha}(M))$ is trivial.
    By \eqref{eqn:kc}, the cokernel is trivial as well. Hence, $\Delta_g$ is Fredholm of index 0. 
\end{proof}

So far, we have assumed that the $S^1$-fiber forms a closed orbit. However, this need not hold in general. In fact, if $(M, g)$ is an ALF manifold that is asymptotically close but not necessarily identical to a model metric $g_0$ at infinity, as described in \Cref{sec:alf}, then the difference $\Delta_g - \Delta_{g_0}$ is a compact operator. As a result, the Fredholm property and its index remain unchanged. This leads to the following:
\begin{corollary}
    Let $(M, g)$ be an ALF manifold. For $\tau \in \RR \setminus \Lambda$, the Laplacian $\Delta_g \colon C_\tau^{2,\alpha}(M) \to C_{\tau+2}^{0, \alpha}(M)$ is Fredholm. Moreover, if $\tau \in (0, m-2)$, then $\Delta_g$ has index 0. 
    \label{cor:lfi}
\end{corollary}

\section{Preservation of ALF structure under Ricci Flow}

We now prove that the ALF structure is preserved under the Ricci flow. The short-time existence of the Ricci flow starting from an ALF metric is guaranteed by \cite[Theorem 1.1]{Shi89}. Combining this with the result established in this section, we conclude that for any ALF metric $g$, there exists a Ricci flow defined on a definite time interval, starting from $g$, whose metrics remain ALF throughout the flow.

\begin{theorem}
    Let $(M, \overline{g})$ be an ALF manifold of order $\eta > 0$ with model metric $g_0$.
    Suppose that there exist $\sigma > \frac{m-2}{2}$ and a metric $g_{\mathrm{RF}}$ on $M$ that is Ricci flat outside a compact set such that
    \[
        \rho_{g_0}^k |\nabla^{g_0, k}(\overline{g}-g_{\mathrm{RF}})|_{g_0} = O(\rho_{g_0}^{-\sigma})
    \]
    for all $k \ge 0$. If $(g(t))_{t \in [0,T]}$ is a Ricci flow with bounded curvature starting from $g(0) = \overline{g}$, then the Ricci flow preserves the decaying order of the metric towards the background Ricci-flat metric, i.e. 
    \[
        \rho_{g_0}^k |\nabla^{g_0, k}(g(t) - g_{\mathrm{RF}})|_{g_0} = O(\rho_{g_0}^{-\sigma})
    \]
    for all $t \in [0,T]$.
    \label{thm:alfp}
\end{theorem}

\begin{proof}
    The argument in the proof of \cite[Theorem 2.2]{Li18} with some modifications yield $|\nabla^{g(t), l} \Rm_{g(t)}| = O(\rho_{g_0}^{-2-l})$. Since we will repeat a similar argument for Ricci curvature with some nontrivial modifications, we omit the proof of the aforementioned decaying rates of $\Rm_{g(t)}$.

    Fix  orthonormal coordinates with respect to $g_0$ in a neighborhood of $p \in M$. 
    In these coordinates, note that 
    \[
        \Ric_{\overline{g}} = \overline{g}^{-1} \ast \partial^2 \overline{g} + \overline{g}^{-1} \ast \overline{g}^{-1} \ast \partial \overline{g} \ast \partial \overline{g}.
    \]
    Therefore comparing with $g_{\mathrm{RF}}$, we have $|\Ric_{\overline{g}}| = O(\rho_{g_0}^{-2-\sigma})$.
    Moreover, by differentiating the above equation repeatedly, we have $|\nabla^{g_0, k} \Ric_{\overline{g}}| = O(\rho_{g_0}^{-2-\sigma-k})$ for all $k \ge 0$.

    Recall the evolution equation of Ricci curvature under Ricci flow: schematically
    \begin{equation}
        \dfrac{\partial}{\partial t} (\Ric_{g(t)}) = \Delta_{g(t)} \Ric_{g(t)} + g(t)^{-1} \ast g(t)^{-1} \ast \Rm_{g(t)} \ast \Ric_{g(t)}.
        \label{eqn:alfpre}
    \end{equation}
    Since the curvature is bounded, there exist $C_1, C_2 >0$ such that for all $t \in [0,T]$,
    \begin{align*}
        \dfrac{\partial}{\partial t} |\Ric_{g(t)}|_{g(t)}^2 &\le \Delta_{g(t)} |\Ric_{g(t)}|_{g(t)}^2 - |\nabla^{g(t)} \Ric_{g(t)}|_{g(t)}^2 + C_1 |\Rm_{g(t)}|_{g(t)}|\Ric_{g(t)}|_{g(t)}^2\\
        &\le \Delta_{g(t)} |\Ric_{g(t)}|_{g(t)}^2 + C_2|\Ric_{g(t)}|_{g(t)}^2.
    \end{align*}
    Define $u(t) = e^{-C_2t}|\Ric_{g(t)}|_{g(t)}^2$ for $t \in [0,T]$ so that $\partial_t u \le \Delta u$ on $M \times [0,T]$, and $h(x) = \rho_{g_0}^{4+2\sigma}$ on $M$. 
    Then $w = hu$ satisfies 
    \[
        (\partial_t - \Delta_{g(t)}) w \le Bw - 2 \langle \nabla^{g(t)} \log h, \nabla^{g(t)} w \rangle_{g(t)}
    \]
    on $M \times [0,T]$ where
    \begin{equation}
    B = \dfrac{2|\nabla^{g(t)}h|_{g(t)}^2 - h \Delta_{g(t)}h}{h^2}.
    \label{eqn:alfp1}
    \end{equation}
    We claim that $B$ is uniformly bounded on $M \times [0,T]$ with respect to $g_0$.

    Since the curvature is bounded and that $\overline{g}$ and $g_0$ are equivalent, there exists $C>0$ such that 
    \begin{equation}
        C^{-1}g_0 \le g(t) \le C g_0
        \label{eqn:alfpcm}
    \end{equation}
    on $M \times [0,T]$.

    Note that 
    \begin{align*}
        \dfrac{\partial}{\partial t} |\nabla^{g(t)} h|_{g_0}^2 &= 2 \left\langle \dfrac{\partial}{\partial t} \nabla^{g(t)} h, \nabla^{g(t)} h \right\rangle_{g_0} \le C_3 \Ric_{g(t)} (\nabla^{g(t)} h, \nabla^{g(t)} h),\\
        \dfrac{\partial}{\partial t} |\Delta_{g(t)} h| &\le \left| \dfrac{\partial}{\partial t} \Delta_{g(t)}h \right|= 2 \left| \langle \Ric_{g(t)}, \nabla^{g(t),2} h \rangle_{g(t)} \right|,
    \end{align*}
    where the first equation follows from Ricci flow equation and second equation follows from \cite[Lemma 2.30]{CLN06}.
    Since the curvature is bounded, by \eqref{eqn:alfpcm}, there exists $C_4 > 1$ such that
    \begin{align*}
        \left| \dfrac{\partial}{\partial t} |\nabla^{g(t)}h|_{g(t)}^2 \right| &\le C_4|\nabla^{g_0} h|_{g_0}^2,\\
        \left| \dfrac{\partial}{\partial t} |\Delta_{g(t)}h| \right| &\le C_4|\nabla^{g_0,2} h|_{g_0}. 
    \end{align*}
    Integrating over $[0,T]$, we get 
    \begin{align*}
        |\nabla^{g(t)}h|_{g(t)}^2 &\le (C_4(T+1)) |\nabla^{g_0} h|_{g_0}^2,\\
        |\Delta_{g(t)}h| &\le (C_4(T+1)) |\nabla^{g_0,2} h|_{g_0}.
    \end{align*}
    By a direct computation on $(M, g_0)$, for some $C_5 > 0$ we have 
    \begin{align*}
        |\nabla^{g_0}h|_{g_0}^2 &\le C_5 \rho_{g_0}^{6+4\sigma},\\
        |\nabla^{g_0,2}h|_{g_0} &\le C_5 \rho_{g_0}^{2+2\sigma}.
    \end{align*}
    Therefore, $B$ as in \eqref{eqn:alfp1} is uniformly bounded on $M \times [0,T]$.
    Then by maximum principle \cite[Theorem 2.1]{Li18}, $|w|$ is uniformly bounded on $M \times [0,T]$, and therefore, $|\Ric_{g(t)}|_{g(t)} \le C \rho_{g_0}^{-2-\sigma}$.

    By induction on $k$, the order of covariant derivative, from \eqref{eqn:alfpre}, we get 
    \begin{align}
        \dfrac{\partial}{\partial t} |\nabla^{k} \Ric_{g(t)}|_{g(t)}^2 
        &= \Delta |\nabla^{k} \Ric_{g(t)}|_{g(t)}^2 - 2|\nabla^{k+1} \Ric_{g(t)}|_{g(t)}^2 + \sum_{l=0}^k \nabla^{l} \Rm_{g(t)} \ast \nabla^{k-l} \Ric_{g(t)} \ast \nabla^k \Ric_{g(t)}\nonumber\\
        &\le \Delta |\nabla^k \Ric_{g(t)}|_{g(t)}^2 + C_6 \sum_{l=0}^k |\nabla^l \Rm_{g(t)}|_{g(t)} |\nabla^{k-l} \Ric_{g(t)}|_{g(t)} |\nabla^k \Ric_{g(t)}|_{g(t)}\label{eqn:alfp2}
    \end{align}
    for some $C_6>0$, where the covariant derivatives are all with respect to $g(t)$.
    
    Now we claim that $|\nabla^k \Ric_{g(t)}| \le C\rho_{g_0}^{-2-k-\sigma}$. We proceed by induction on $k$. The claim for $k=0$ is proven above. Suppose that the claim is true when $k$ is substituted by any integer less than $k$. 
    Define $h_k = \rho_{g_0}^{4+2k+2\sigma}$ and $w_k = h_k|\nabla^k \Ric_{g(t)}|_{g(t)}^2$. 
    By \eqref{eqn:alfp2}, 
    \[
        (\partial_t - \Delta_{g(t)})w_k \le B_k w_k - 2\langle \nabla \log h_k, \nabla w_k \rangle + C_6 \sum_{l=0}^k h_k |\nabla^l \Rm_{g(t)}|_{g(t)} |\nabla^{k-l} \Ric_{g(t)}|_{g(t)} |\nabla^k \Ric_{g(t)}|_{g(t)},
    \]
    where 
    \[
        B_k = \dfrac{2|\nabla h_k|^2-h_k\Delta h_k}{h_k^2}.
    \]
    As before, $B_k$ is uniformly bounded on $M \times [0,T]$. By inductive hypothesis and decaying rates for $|\nabla^l \Rm_{g(t)}|$, 
    \[
        h_k |\nabla^l \Rm_{g(t)}|_{g(t)} |\nabla^{k-l} \Ric_{g(t)}|_{g(t)} |\nabla^k \Ric_{g(t)}|_{g(t)} \le C_7 \rho_{g_0}^{k+\sigma} |\nabla^k \Ric_{g(t)}|_{g(t)} \le C_8 w_k^{1/2}
    \]
    for some $C_8 > 0$ if $0 < l \le k$ and 
    \[
        h_k|\nabla^l \Rm_{g(t)}|_{g(t)} |\nabla^{k-l} \Ric_{g(t)}|_{g(t)} |\nabla^k \Ric_{g(t)}|_{g(t)} = h_k |\Rm_{g(t)}|_{g(t)} |\nabla^k \Ric_{g(t)}|_{g(t)}^2 \le C_9 w_k^2
    \]
    if $l=k$. Hence for $C_{10} = kC_8+C_9$, 
    \[
        (\partial_t - \Delta)w_k \le -2\langle \nabla \log h, \nabla w_k \rangle + C_{10}(w_k + w_k^{1/2}).
    \]
    Since the solution to $\frac{dv}{dt} = C_{10}(v+v^{1/2})$ with $v(0)=c>0$ is bounded, by the maximum principle, $|w_k|$ is bounded. Therefore, $|\nabla^k \Ric_{g(t)}|_{g(t)} = O(\rho_{g_0}^{-2-k-\sigma})$ on $M \times [0,T]$.

    For any $k \ge 0$, 
    \[
        \partial_t(\nabla^{g_0, k}g(t)) = \nabla^{g_0, k} (\partial_t g(t)) = -2 \nabla^{g_0, k} \Ric_{g(t)}.
    \]
    Then for any $t \in [0,T]$ and any vector field $X$ on $M$, 
    \begin{align*}
        |\log \nabla^{g_0, k} g(t)(X,X) - \log \nabla^{g_0, k} g(0)(X,X)| = \left| \int_0^t \dfrac{-2\nabla^{g_0, k} \Ric_{g(t)}(X,X)}{g(s)(X,X)} ds \right|
        \le C_{11} \rho_{g_0}^{-2-k-\sigma},
    \end{align*}
    where $C_{11} > 0$ is a uniform constant on $M \times [0,T]$ independent on $X$. Therefore, 
    \[
        |\nabla^{g(0), k} (g(t)-g(0))| = g(0) O(\rho_{g_0}^{-2-k-\sigma})
    \]
    and since $g(0)$ is equivalent to $g_0$, we get the desired decaying rate.
\end{proof}

\section{A renormalized Perelman's $\lambda$-functional $\lambda_{\ALF}$}

In this section, we define an adapted $\lambda$-functional, called $\lambda_{\ALF}$. We begin by introducing Haslhofer's $\lambda$-functional, denoted by $\lambda_{\ALF}^\circ$. A Hardy-type inequality, different from the one used in \cite{DO20}, is required to justify the existence of this functional, since the inequality used in \cite{DO20} relies on a non-collapsing assumption that does not hold in our setting. We first define $\lambda_{\ALF}^\circ$ as the infimum of a functional called $\widetilde{\cF}$, and then show that this infimum is attained by a unique positive function with appropriate decay. The existence and uniqueness follow from the Fredholm theory developed in \Cref{sec:fl}.

\subsection{A preliminary construction of $\lambda_{\ALF}$ under scalar curvature conditions}
Throughout this section, $\epsilon > 0$ is a small enough constant.

\begin{assumption}
    \begin{enumerate}
        \item $m \ge 3$, $\eta >0$, $\sigma \in (\frac{m-2}{2}, m-2)$, and $\alpha \in (0,1)$ are given.
        \item $(M^{m+1}, g)$ is an ALF manifold of order $\eta$ with a model metric $g_0$.
        \item There exists a Ricci-flat metric $g_{\mathrm{RF}}$ for which 
        \[
            \rho_{g_0}^k |\nabla^{g_0, k}(g-g_{\mathrm{RF}})|_{g_0} = O(\rho_{g_0}^{-\sigma}).
        \]
        \item Given $\epsilon > 0$, a preliminary neighborhood of $g_{\mathrm{RF}}$ that we consider is 
        \begin{equation}
            \cM_{\sigma}^{2,\alpha}(g_{\mathrm{RF}}, \epsilon) := \{ g' \in B_\epsilon(g_{\mathrm{RF}}) ; C_\sigma^{2,\alpha} \;|\; R_{g'} = O(\rho_{g_0}^{-\eta'}) \text{ for } \eta' > m \text{ or } R_{g'} \ge 0 \}.
            \label{eqn:nbhd}
        \end{equation}
    \end{enumerate}
    \label{ass:bs}
\end{assumption}
\begin{definition}
    Let $(M, g)$ be an ALF manifold of order $\eta > \frac{m-2}{2}$. For $w-1 \in C_c^\infty(M)$, define 
    \[
        \widetilde{\cF}_{\ALF}(g,w) = \int_M (4|\nabla^g w|_g^2 + R_g w^2).
    \]
    Then \emph{$\lambda_{\ALF}^\circ$-functional} is
    \[
        \lambda_{\ALF}^\circ(g) = \inf_{\substack{w\colon M \to \RR \\ w-1 \in C_c^\infty(M)}} \widetilde{\cF}_{\ALF}(g,w).
    \]
\end{definition}

\begin{proposition}
    Let $(M, g$) be an ALF manifold with $m$, $\alpha$, $\eta$, $\sigma$, $g_0$, $g_{\mathrm{RF}}$ as in \Cref{ass:bs}.
    Then there exists a unique solution $w_g \in C_\sigma^{2,\alpha}(M)$ to 
    \begin{equation}
        -4\Delta_g w_g + R_g w_g = 0.
        \label{eqn:lcalf2}
    \end{equation}
    If $g \in \cM_{\sigma}^{2,\alpha}(g_{\mathrm{RF}}, \epsilon)$,
    then the infimum of $\lambda_{\ALF}^\circ$ is attained by $w_g$.
    \label{prop:lcalf}
\end{proposition}

To justify the functional, we will need the following Hardy-type inequality.
\begin{theorem}[{\cite[Theorem 1.2]{Hei11}}]
    Let $(M, g)$ be an ALF manifold of order $\eta > 0$. Then there exists $C_H>0$ such that
    \begin{equation}
        \int_M |\nabla^{\overline{g}}\phi|_{g}^2 \ge C_H\int_M \dfrac{|\phi|^2}{\rho_{g_0}^2}  
        \label{eqn:hi}
    \end{equation}
    for all $\phi \in C_c^\infty(M)$.
    \label{thm:hi}
\end{theorem}

\begin{proof}[Proof of \Cref{prop:lcalf}]
    We claim that $-4\Delta_g + R_g \colon C_\sigma^{2,\alpha}(M) \to C_{\sigma+2}^{0,\alpha}(M)$ is an isomorphism. 
    Since $R_g \colon C_\sigma^{2,\alpha}(M) \to C_{\sigma+2}^{0,\alpha}(M)$ is a compact operator, by \Cref{cor:lfi}, $-4\Delta_g + R_g$ is a Fredholm operator of index 0. Then \eqref{eqn:lcalf3} ensures that the operator is indeed an isomorphism. 

    Since $g \in B_\epsilon(g_{\mathrm{RF}}; C_\sigma^{2,\alpha})$, $R_g \in C_{\sigma+2}^{0,\alpha}(M)$.
    Therefore, there exists a unique solution $v_g \in C_\sigma^{2,\alpha}(M)$ such that $-4\Delta_g v_g + R_g v_g = -R_g$. Then $w_g = 1+v_g$ is a unique solution to \eqref{eqn:lcalf2}.

    Assume from now that $g \in \cM_{\sigma}^{2,\alpha}(g_{\mathrm{RF}}, \epsilon)$.
    We first show that $\lambda_{\ALF}^\circ(g)$ is finite. Suppose that $g$ has integrable scalar curvature.
    By density argument, \eqref{eqn:hi} holds for $\phi \in C_\tau^{2,\alpha}(M)$ with $\tau \in (\frac{m-2}{2}, m-2)$.
    Note that for small enough $\epsilon > 0$, the metrics in $\cM_{\sigma}^{2,\alpha}(g_{\mathrm{RF}},\epsilon)$ are uniformly equivalent, and therefore \eqref{eqn:hi} uniformly holds for all metrics in $\cM_{\sigma}^{2,\alpha}(g_{\mathrm{RF}},\epsilon)$ for some fixed $C_H>0$ given $\epsilon$ is small enough. Together with the absorbing inequality, we get
    \begin{align*}
        \int_M (4|\nabla^g w|_g^2+R_gw^2)
        &= \int_M \left(4|\nabla^g(w-1)|_g^2 + R_g((w-1)+1)^2\right)\\
        &\ge 4C_H \int_M \dfrac{(w-1)^2}{\rho_{g_0}^2} - 2 \left(   \int_M |R_g|(w-1)^2 + \int_M |R_g| \right)\\
        &\ge 4C_H \int_M \dfrac{(w-1)^2}{\rho_{g_0}^2} - 2 \left(\sup_M |R_g|\rho_{g_0}^2\right) \int_M \dfrac{(w-1)^2}{\rho_{g_0}^2} - 2 \int_M |R_g|\\
        &\ge -2\int_M |R_g|,
    \end{align*}
    given $\norm{g - g_{\mathrm{RF}}}_{C_\sigma^{2,\alpha}(M)} < \epsilon$ with $\epsilon>0$ small enough.
    
    A similar argument also shows that
    \begin{align*}
        \int_M (4 |\nabla^g \phi|_g^2 + R_g\phi^2)
        &\ge \int_M 4 |\nabla^g\phi|_g^2 - \sup_M(|R_g|\rho_{g_0}^2) \int_M \dfrac{\phi^2}{\rho_{g_0}^2}\\
        &\ge \int_M 4 |\nabla^g\phi|_g^2 - C_H \sup_M(|R_g|\rho_{g_0}^2) \int_M |\nabla^g \phi|_g^2\\
        &\ge \int_M |\nabla^g \phi|_g^2.
    \end{align*}

    As a consequence, we get 
    \begin{equation}
        \langle -4\Delta_g \phi + R_g \phi, \phi \rangle \ge \norm{\nabla^g \phi}_{L^2}^2
        \label{eqn:lcalf3}
    \end{equation}
    for all $\phi \in C_c^\infty(M)$. By a density argument, \eqref{eqn:lcalf3} holds for $\phi \in C_\tau^{2,\alpha}(M)$ as well.

    Note that if $g$ has nonnegative scalar curvature, then we directly get that $\lambda_{\ALF}^\circ(g) \ge 0$ and \eqref{eqn:lcalf3} holds.

    Now we prove that $\lambda_{\ALF}^\circ$ is attained by $w_g$. To that end, it is enough to prove 
    \[
        \int_M (4|\nabla^g(w_g+\phi)|_g^2 + R_g(w_g+\phi)^2) \ge \int_M (4|\nabla^gw_g|_g^2 + R_gw_g^2)
    \]
    for all $\phi \in C_c^\infty(M)$. This is equivalent to
    \[
        \int_M 2(-4\Delta_g w_g + R_g w_g)\phi + (4|\nabla^g\phi|_g^2+R_g\phi^2) \ge 0,
    \]
    which follows from \eqref{eqn:lcalf2} and \eqref{eqn:lcalf3}.
\end{proof}

\begin{proposition}
    Let $(M, g$) be an ALF manifold with $m$, $\alpha$, $\eta$, $\sigma$, $g_0$, $g_{\mathrm{RF}}$ as in \Cref{ass:bs}.
    Suppose $g \in \cM_{\sigma}^{2,\alpha}(g_{\mathrm{RF}}, \epsilon)$ and
    let $w_g$ be as in \Cref{prop:lcalf}. Then $w_g$ is positive.
\end{proposition}

\begin{proof}
    Note that $|w_g|$ is a weak solution to \eqref{eqn:lcalf2}. Then elliptic regularity gives $|w_g| \in C_{\loc}^{2,\alpha}(M)$, and then Kato inequality gives $|w_g| \in C_\sigma^{2,\alpha}(M)$. Since the equation has a unique solution, $w_g = |w_g|$, i.e. $w_g$ is nonnegative. Now define 
    \[
        W_g(x,t) = \exp\left( \dfrac{\sup_M R_g}{4} t \right) w_g(x),
    \]
    then a direct computation gives that 
    \[
        (\partial_t - \Delta_g)W_g \ge 0,
    \]
    so by strong maximum principle, $w_g$ is positive.
\end{proof}

We will perform multiple integration by parts. To simplify notation, we introduce the following shorthand: although these are not intrinsic balls, we denote
\begin{align*}
    B_R^{g_0} &:= \{\rho_{g_0} < R\} \subset M,\\
    \partial B_R^{g_0} &:= \{\rho_{g_0} = R\} \subset M.
\end{align*}

\begin{proposition}
    Let $(M, g$) be an ALF manifold with $m$, $\alpha$, $\eta$, $\sigma$, $g_0$, $g_{\mathrm{RF}}$ as in \Cref{ass:bs}.
    Suppose $g \in \cM_{\sigma}^{2,\alpha}(g_{\mathrm{RF}}, \epsilon)$ and
    let $w_g$ be as in \Cref{prop:lcalf}. Then 
    \[
        \lambda_{\mathrm{ALF}}^\circ(g) = \int_M (4|\nabla^g w_g|_g^2 + R_gw_g^2) = \int_M R_gw_g.
    \]
    \label{prop:rw}
\end{proposition}

\begin{proof}
    The first equality is proven in \Cref{prop:lcalf}.
    For the second equality, we use integration by parts multiple times and observe $|(w_g-1)\nabla^g w_g|$ decays fast:
    \begin{align*}
        \int_M (4|\nabla^g w_g|_g^2 + R_gw_g^2)
        &= \lim_{R \to \infty} \left[\int_{ B_R^{g_0} } (-4\Delta_g w_g + R_gw_g)w_g + 4\int_{\partial B_R^{g_0}} \langle \nabla^g w_g, n_g \rangle_g w_g \right]\\
        &= \lim_{R \to \infty} \left[ 4\int_{\partial B_R^{g_0}} \left(\langle \nabla^g w_g, n_{g_0} \rangle_g + \langle (w_g-1)\nabla^g w_g, n_g \rangle_g \right) \right]\\
        &= \lim_{R \to \infty}4 \int_{B_R^{g_0}} \Delta_g w_g
        = \lim_{R \to \infty} \int_{B_R^{g_0}} R_gw_g
        = \int_M R_gw_g.
    \end{align*}
\end{proof}

\subsection{Relative mass and $\lambda_{\ALF}$ beyond scalar curvature assumptions}
Unlike the ALE setting, where we have the ADM mass, there is no universally accepted analogue in the ALF setting. By choosing a suitable reference metric, we introduce a relative mass, which serves the role of the ADM mass for our purposes.

\begin{proposition}
    Let $(M, g$) be an ALF manifold with $m$, $\alpha$, $\eta$, $\sigma$, $g_0$, $g_{\mathrm{RF}}$ as in \Cref{ass:bs}.
    If $g, g+h \in \cM_{\sigma}^{2,\alpha}(g_{\mathrm{RF}}, \epsilon)$, then the limit defining \emph{relative mass}
    \begin{equation}
        m(g+h, g) := \lim_{R \to \infty} \int_{\partial B_R^{g_0}} \langle \div_g h - \nabla^g \Tr_g h, n_g \rangle_{g}
        \label{eqn:rm}
    \end{equation}
    exists.
    \label{prop:rme}
\end{proposition}

\begin{proof}
    Recall that the first variation of scalar curvature is 
    \[
        \delta_g R_g(h) = \div_g(\div_g h) - \Delta_g \Tr_g h - \langle h, \Ric_g \rangle_g.
    \]
    Integrating the formula, we get
    \begin{align}
        R_{g+h} - R_g &= \int_0^1 \delta_{g+th} R_{g+th}(h) dt\nonumber\\
        &= \int_0^1 (\div_{g+th}\div_{g+th}h-\Delta_{g+th}\Tr_{g+th}h - \langle h, \Ric_{g+th} \rangle_{g+th})dt\nonumber\\
        &=: \div_g \div_g h - \Delta_g \Tr_g h - \langle h, \Ric_g \rangle_g + Q_g(h). \label{eqn:rme1}
    \end{align}
    Since we have
    \begin{align*}
        &(\div_{g+th}\div_{g+th}h - \Delta_{g+th}\Tr_{g+th}h) - (\div_g\div_g h - \Delta_g \Tr_g h) = O(\rho_{g_0}^{-2\sigma-2}),\\
        &\langle h, \Ric_{g+th} \rangle_{g+th} - \langle h, \Ric_g \rangle_g = O(\rho_{g_0}^{-2\sigma-2}),
    \end{align*}
    we conclude that $Q_g(h) \in O(\rho_{g_0}^{-2\min(\eta,\sigma)-2})$. In particular, $Q_g(h)$ is integrable. 

    By Integration by parts,
    \begin{align*}
        m(g+h, g) &= \lim_{R \to \infty} \int_{B_R^{g_0}} \div_g(\div_g h - \nabla^g \Tr_g h)\\
        &= \lim_{R \to \infty} \int_{B_R^{g_0}} \left[ R_{g+h} - R_g + \langle h, \Ric_g\rangle_g - Q_g(h) \right],
    \end{align*}
    and the limit exists by \eqref{eqn:nbhd} and  $\langle h, \Ric_g \rangle_g = O(\rho_{g_0}^{-2\sigma-2})$.
\end{proof}

\begin{definition}
    Let $(M, g$) be an ALF manifold with $m$, $\alpha$, $\eta$, $\sigma$, $g_0$, $g_{\mathrm{RF}}$ as in \Cref{ass:bs}. For $f \in C_\tau^{2,\alpha}(M)$ with $\tau \in (\frac{m-2}{2}, m-2)$, define 
    \[
        \cF_{\ALF}(g,f) = \int_M (|\nabla^g f|_g^2 + R_g)e^{-f}.
    \]
\end{definition}

\begin{remark}
    $\cF_{\ALF}(g,f) = \widetilde{\cF}_{\ALF}(g, e^{-f/2})$.
\end{remark}

\begin{proposition}
    Let $(M, g$) be an ALF manifold with $m$, $\alpha$, $\eta$, $\sigma$, $g_0$, $g_{\mathrm{RF}}$ as in \Cref{ass:bs}.
    For $\tau \in (\frac{m-2}{2}, m-2)$, $(g,f) \in \cM_{\sigma}^{2,\alpha}(g_{\mathrm{RF}}, \epsilon) \times C_\tau^{2,\alpha}(M)$, and $(h, \phi) \in C_\sigma^{2,\alpha}(S^2T^\ast M) \times C_\tau^{2,\alpha}(M)$ with $g+h \in \cM_{\sigma}^{2,\alpha}(g_{\mathrm{RF}}, \epsilon)$, 
    \begin{align}
        \delta_{g,f} \cF_{\ALF}(h, \phi) &= -\int_M \langle\Ric_f, h\rangle_g e^{-f} 
        + \int_M R_f\left( \frac{1}{2}\Tr_g h - \phi \right)e^{-f}
        + m(g+h,g), \label{eqn:vff}\\
        \delta_g \left[\lambda_{\ALF}^\circ(g)\right](h) &= -\int_M \langle \Ric_{f_g},h\rangle_g e^{-f_g} + m(g+h,g). \label{eqn:vfl}
    \end{align}
    \label{prop:vf}
\end{proposition}

We record here several variation formulas that will be used in the proof.
\begin{lemma}
    \begin{align}
        \delta_{g} R_g (h) &= \div_g (\div_g h) -\Delta_g \Tr_g h - \langle h, \Ric_g \rangle_g,\label{eqn:vfs1}\\
        \delta_{g,f}|\nabla^g f|_g^2 (h,\phi) &= -h(\nabla^g f, \nabla^g f) + 2\langle \nabla^g f, \nabla^g \phi \rangle_g,\\
        \delta_{g,f} [e^{-f}dV_g] (h,\phi) &= \left( \frac{1}{2}\Tr_gh - \phi \right)e^{-f}dV_g.
    \end{align}
    \label{lem:vfs}
\end{lemma}

\begin{proof}[Proof of \Cref{prop:vf}]
    By \Cref{lem:vfs},
    \begin{align*}
        &\delta_{g,f}[(|\nabla^g f|_g^2 + R_g)e^{-f}dV_g](h,\phi)\\
        &\quad = \left[ -h(\nabla^g f, \nabla^g f) + (\div_g (\div_g h) - \Delta_g \Tr_g h - \langle \Ric_g, h \rangle_g) + 2\langle \nabla^g f, \nabla^g \phi \rangle_g\right]e^{-f}dV_g\\
        &\quad\quad + (|\nabla^g f|_g^2+R_g) \left( \frac{1}{2}\Tr_g h - \phi \right) e^{-f}dV_g.
    \end{align*}

    By Integration by parts:
    \begin{align*}
        \int_{B_R^{g_0}} \div_g(\div_g h)e^{-f}
        &= \int_{B_R^{g_0}} (\div_g h)(\nabla^g f)e^{-f} + \int_{\partial B_R^{g_0}}\langle \div_g h, n_g \rangle_g e^{-f}\\
        &= \int_{B_R^{g_0}} \left( h(\nabla^g f, \nabla^g f) - \langle h, \nabla^{g,2}f \rangle_g \right)e^{-f}
        + \int_{\partial B_R^{g_0}} \langle \div_g h + h(\nabla^g f), n_g \rangle_g e^{-f},
        \\
        \int_{B_R^{g_0}} (\Delta_g \Tr_g h)e^{-f}
        &= \int_{B_R^{g_0}} \langle \nabla^g \Tr_g h, \nabla^g f \rangle_g e^{-f} + \int_{\partial B_R^{g_0}} \langle \nabla^g \Tr_g h, n_g \rangle_g e^{-f}\\
        &= -\int_{B_R^{g_0}} (\Tr_g h)(\Delta_g f - |\nabla^g f|_g^2)e^{-f}
        + \int_{\partial B_R^{g_0}} \langle \nabla^g \Tr_g h + (\Tr_g h)\nabla^g f, n_g \rangle_g e^{-f},
        \\
        \int_{B_R^{g_0}} \langle \nabla^g f, \nabla^g \phi \rangle_g e^{-f}
        &= -\int_{B_R^{g_0}} (\Delta_g f - |\nabla^g f|_g^2)\phi e^{-f} + \int_{\partial B_R^{g_0}} \langle \nabla^g f, n_g \rangle_g \phi e^{-f}.
    \end{align*}

As a consequence, we get 
\begin{align}
    &\int_{B_R^{g_0}} \delta_{g,f}\left[ (|\nabla^g f|_g^2+R_g)e^{-f}dV_g \right](h,\phi)\label{eqn:vfb}\\
    &= \int_{B_R^{g_0}} -\langle \Ric_{f_g}, h\rangle_g e^{-f}
    + \int_{B_R^{g_0}} (2\Delta_g f - |\nabla^g f|_g^2+R_g)\left( \frac{1}{2}\Tr_g h - \phi \right)e^{-f}\nonumber\\
    &\quad + \int_{\partial B_R^{g_0}} \langle \div_g h - \nabla^g \Tr_g h, n_g \rangle e^{-f}
    + \int_{\partial B_R^{g_0}} \langle h(\nabla^g f) - (\Tr_g h) \nabla^g f + 2\phi\nabla^g f, n_g \rangle_g e^{-f}. \nonumber
\end{align}
Since 
\begin{equation}
    h (\nabla^g f) - (\Tr_gh)\nabla^g f + 2\phi \nabla^g f = O(\rho_{g_0}^{-\sigma-\tau-2}),
    \label{eqn:vfu}
\end{equation}
we conclude \eqref{eqn:vff}.

To compute the variation of $\lambda_{\ALF}^\circ$, we first define 
\begin{equation}
    f_g = -2\log w_g
    \label{eqn:fg}
\end{equation}
where $w_g$ is as in \Cref{prop:lcalf}. 
Then following \eqref{eqn:lcalf2},
\begin{equation}
    R_{f_g} = 2\Delta_g f_g - |\nabla^g f_g|_g^2 + R_g = 0.
    \label{eqn:fge}
\end{equation}
Since $\lambda_{\ALF}^\circ(g) = \widetilde{\cF}_\ALF(g, w_g) = \cF_\ALF(g, f_g)$,
\[
    \delta_g \left[\lambda_{\ALF}^\circ(g)\right](h) = \delta_{g, f_g} \cF_\ALF(h, \delta_g[f_g](h)),
\]
so together with \eqref{eqn:fge}, \eqref{eqn:vfl} follows.
\end{proof}

\begin{definition}
    Let $(M, g$) be an ALF manifold with $m$, $\alpha$, $\eta$, $\sigma$, $g_0$, $g_{\mathrm{RF}}$ as in \Cref{ass:bs}.
    For $g \in \cM_{\sigma}^{2,\alpha}(g_{\mathrm{RF}}, \epsilon)$, define a renormalized Perelman's $\lambda$-functional, denoted by $\lambda_{\ALF}(g, \overline{g})$, by 
    \[
        \lambda_{\ALF}(g, \overline{g}) = \lambda_{\ALF}^\circ(g) - m(g, \overline{g}).
    \]
\end{definition}

An important property of relative mass is its additivity: changing the choice of reference metric shifts the relative mass only by an additive constant, and consequently shifts $\lambda_{\ALF}$ by negative of the same constant.
\begin{proposition}[Additivity of relative mass]
    Let $(M, \overline{g}$) be an ALF manifold with $m$, $\alpha$, $\eta$, $\sigma$, $g_0$, $g_{\mathrm{RF}}$ as in \Cref{ass:bs}.
    Then for $g_1, g_2, g_3 \in \cM_{\sigma}^{2,\alpha}(g_{\mathrm{RF}}, \epsilon)$, 
    \[
        m(g_1, g_2) + m(g_2, g_3) = m(g_1, g_3).
    \]
    \label{prop:rml}
\end{proposition}
\begin{proof}
    We show that the difference between the relative masses of $g_1$ and $g_2$ taken relative to $g_3$ equals the relative mass of $g_1$ with respect to $g_2$.
    \begin{align*}
        m(g_1, g_3) - m(g_2, g_3)
        &= \lim_{R \to \infty} \left[ \int_{\partial B_R^{g_0}} \langle \div_{g_3}(g_1-g_3) - \nabla^{g_3} \Tr_{g_3}(g_1-g_3), n_g \rangle_{g_3} \right.\\
        &\qquad\qquad\qquad \left. - \int_{\partial B_R^{g_0}} \langle \div_{g_3}(g_2-g_3) - \nabla^{g_3}\Tr_{g_3}(g_2-g_3), n_g \rangle_{g_3} \right]\\
        &= \lim_{R \to \infty} \int_{\partial B_R^{g_0}} \langle \div_{g_3}(g_1-g_2) - \nabla^{g_3}\Tr_{g_3}(g_1-g_2), n_g \rangle_{g_3}\\
        &= \lim_{R \to \infty} \int_{\partial B_R^{g_0}} \langle \div_{g_2}(g_1-g_2) - \nabla^{g_2}\Tr_{g_2}(g_1-g_2), n_g \rangle_{g_2} \\
        &= m(g_1, g_2).
    \end{align*}
    In the last two lines, we used $g_1-g_2, g_2-g_3 \in C_\sigma^{2,\alpha}(S^2T^\ast M)$.
\end{proof}

\begin{proposition}
    Let $(M, g$) be an ALF manifold with $m$, $\alpha$, $\eta$, $\sigma$, $g_0$, $g_{\mathrm{RF}}$ as in \Cref{ass:bs}. 
    If $g \in \cM_{\sigma}^{2,\alpha}(g_{\mathrm{RF}}, \epsilon)$ has integrable scalar curvature and $(g(t))_{t \in [0,T]}$ with $g(0) = g$ is a Ricci flow satisfying the assumption in \Cref{thm:alfp}, then the relative mass $m(g(t), g_{\mathrm{RF}})$ is preserved on $t \in [0,T]$.
\end{proposition}
\begin{proof}
    By \Cref{thm:alfp} and the contracted Bianchi identity, we have 
    \begin{align*}
        \dfrac{d}{dt} m(g(t), g_{\mathrm{RF}})
        &= \lim_{R \to \infty} \int_{\partial B_R^{g_0}} \langle \div_{g_{\mathrm{RF}}} (\partial_t g(t)) - \nabla^{g_{\mathrm{RF}}} \Tr_{g_{\mathrm{RF}}} (\partial_t g(t)), n_{g_{\mathrm{RF}}} \rangle_{g_{\mathrm{RF}}}\\
        &= \lim_{R \to \infty} \int_{\partial B_R^{g_0}} \langle \div_{g(t)} (\partial_t g(t)) - \nabla^{g(t)} \Tr_{g(t)} (\partial_t g(t)), n_{g(t)}\rangle_{g(t)}\\
        &= \lim_{R \to \infty} \int_{\partial B_R^{g_0}} \div_{g(t)} (-2\Ric_{g(t)}) - \nabla^{g(t)} \Tr_{g(t)}(-2\Ric_{g(t)}), n_{g(t)}\rangle_{g(t)}\\
        &= \lim_{R \to \infty} \int_{\partial B_R^{g_0}} \langle \nabla^{g(t)} R_{g(t)}, n_{g(t)} \rangle_{g(t)}.
    \end{align*}
    The proof of \cite[Lemma 11]{MS12} with slight modification where one take parabolic cylinders of size a small fraction of fiber length on the model space when apply the local maximum principle, $L^p$ estimates, and the Sobolev inequalities gives 
    \[
        \lim_{R \to \infty} \int_{\partial B_R^{g_0}} \langle \nabla^{g(t)} R_{g(t)}, n_{g(t)}\rangle_{g(t)} = 0
    \]
    when $t>0$. The constancy of the relative mass extends to $t=0$ by noting that the relative mass is continuous because of the bound on the curvature.
\end{proof}

We now proceed to show that $\lambda_{\ALF}$ is defined not only for metrics with nonnegative or integrable scalar curvature, but also for a broader class of metrics.

\begin{proposition}
    Let $(M, g$) be an ALF manifold with $m$, $\alpha$, $\eta$, $\sigma$, $g_0$, $g_{\mathrm{RF}}$ as in \Cref{ass:bs}. 
    The functional $\lambda_{\ALF}(g, g_{\mathrm{RF}})$, initially defined on $\cM_{\sigma}^{2,\alpha}(g_{\mathrm{RF}}, \epsilon)$, extends to $B_\epsilon(g_{\mathrm{RF}};C_\sigma^{2,\alpha})$ as
    \[
        \lambda_{\ALF}(g, g_{\mathrm{RF}})= \lim_{R \to \infty} \left[ \int_{B_R^{g_0}} R_gw_g - \int_{\partial B_R^{g_0}} \langle\div_{g_{\mathrm{RF}}}(g-g_{\mathrm{RF}})-\nabla^{g_{\mathrm{RF}}}\Tr_{g_{\mathrm{RF}}}(g-g_{\mathrm{RF}}), n_{g_{\mathrm{RF}}}\rangle_{g_{\mathrm{RF}}}\right],
    \]    
    where $w_g$ is as in \Cref{prop:bs}.
    \label{prop:le}
\end{proposition}
\begin{proof}
    If $g \in \cM_{\sigma}^{2,\alpha}(g_{\mathrm{RF}}, \epsilon)$, then by \Cref{prop:lcalf}, \Cref{prop:rw}, and the proof of \Cref{prop:rme} with $Q_g(h)$ as in \eqref{eqn:rme1},
    \begin{align}
        \lambda_{\ALF}(g, g_{\mathrm{RF}})
        &= \lambda_{\ALF}^\circ(g) - m(g,g_{\mathrm{RF}})\nonumber\\
        &= \lim_{R \to \infty} \left[ \int_{B_R^{g_0}} R_gw_g - \int_{\partial B_R^{g_0}} \langle \div_{g_{\mathrm{RF}}} (g-g_{\mathrm{RF}}) - \nabla^{g_{\mathrm{RF}}} \Tr_{g_{\mathrm{RF}}} (g-g_{\mathrm{RF}}), n_{g_{\mathrm{RF}}} \rangle_{g_{\mathrm{RF}}} \right]\nonumber\\
        &= \lim_{R \to \infty} \left[ \int_{B_R^{g_0}} R_gw_g - \div_{g_{\mathrm{RF}}} (\div_{g_{\mathrm{RF}}}(g-g_{\mathrm{RF}})-\nabla^{g_{\mathrm{RF}}}\Tr_{g_{\mathrm{RF}}}(g-g_{\mathrm{RF}})) \right]\nonumber\\
        &= \lim_{R \to \infty} \left[ \int_{B_R^{g_0}} R_g(w_g-1) + R_{g_{\mathrm{RF}}} - \langle g-g_{\mathrm{RF}}, \Ric_{g_{\mathrm{RF}}} \rangle_{g_{\mathrm{RF}}} - Q_{g_{\mathrm{RF}}}(g-g_{\mathrm{RF}}) \right]. \label{eqn:le1}
    \end{align}
    In the second and third lines, we implicitly use the fact that the difference of integrals with respect to area measure of $g$ and $g_{\mathrm{RF}}$ is negligible as $R \to \infty$ thanks to the fast enough decay of $g-g_{\mathrm{RF}}$. 
    Since \eqref{eqn:le1} is well defined for $g \in B_\epsilon(g_{\mathrm{RF}};C_\sigma^{2,\alpha})$, we would extend $\lambda_{\ALF}$ by this equality.
\end{proof}

    An argument close to the proof of \Cref{prop:rw} yields
    \begin{align*}
        \lim_{R \to \infty} \left[ \int_{B_R^{g_0}} R_gw_g - \int_{B_R^{g_0}} (|\nabla^g f_g|_g^2 + R_g)e^{-f_g} \right]
        &= \lim_{R \to \infty} \left[\int_{B_R^{g_0}} R_gw_g - \int_{B_R^{g_0}} (4|\nabla^g w_g|_g^2+R_gw_g^2) \right] = 0,
    \end{align*}
    where $f_g = -2\log w_g$. Therefore, we may rewrite the functional as 
    \begin{equation}
        \lambda_{\ALF}(g, g_{\mathrm{RF}}) = \lim_{R \to \infty} \left[ \int_{B_R^{g_0}} (|\nabla^g f_g|_g^2 + R_g)e^{-f_g} - \int_{\partial B_R^{g_0}} \langle\div_{g_{\mathrm{RF}}}(g-g_{\mathrm{RF}})-\nabla^{g_{\mathrm{RF}}}\Tr_{g_{\mathrm{RF}}}(g-g_{\mathrm{RF}}), n_{g_{\mathrm{RF}}}\rangle_{g_{\mathrm{RF}}} \right].
        \label{eqn:lalfe}
    \end{equation}

\subsection{Monotonicity of $\lambda_{\ALF}$ along Ricci Flow}

We show that Ricci flow is a gradient flow of $\lambda_{\ALF}$ in a weighted $L^2$-sense. 
\begin{theorem}
    Let $(M, g$) be an ALF manifold with $m$, $\alpha$, $\eta$, $\sigma$, $g_0$, $g_{\mathrm{RF}}$ as in \Cref{ass:bs}. If $(g(t))_{t \in [0, T]} \subset B_\epsilon(g_{\mathrm{RF}};C_\sigma^{2,\alpha})$ is a Ricci Flow on $M$, then 
    \begin{equation}
        \dfrac{d}{dt} \lambda_{\ALF}(g(t), g_{\mathrm{RF}}) = 2\norm{ \Ric_{g(t)} + \nabla^{g(t),2}f_{g(t)}}_{L^2(e^{-f_{g(t)}})}^2.
        \label{eqn:mlrf}
    \end{equation}
    Moreover, $\lambda_{\ALF}(g(t), g_{\mathrm{RF}})$ is constant for Ricci-flat metrics only.
    \label{thm:mlrf}
\end{theorem}

This implies the usual consequences such as no-breather theorems.

\begin{lemma}
    Let $(M, g$) be an ALF manifold with $m$, $\eta$, $\sigma$, $g_0$, $g_{\mathrm{RF}}$ as in \Cref{ass:bs}. If $(g(t))_{t \in [0, T]} \subset B_\epsilon(g_{\mathrm{RF}};C_\sigma^{2,\alpha})$ is a continuous 1-parameter family of metrics on $M$, then there exists a uniform constant $C>0$ independent on $t$ such that 
    \[
        \norm{f_{g(t)}}_{C_\sigma^{2,\alpha}(M)} \le C
    \]
    for all $t \in [0,T]$.
    \label{lem:fgub}
\end{lemma}

\begin{proof}
    Recall that $f_{g(t)}$ is given by 
    \begin{equation}
        f_{g(t)} = -2\log(1+v_{g(t)}),
        \label{eqn:fgub1}
    \end{equation}
    where $v_{g(t)}$ is the unique solution to 
    \begin{equation}
        (-4\Delta_{g(t)}+R_{g(t)})v_{g(t)} = -R_{g(t)}.
        \label{eqn:fgub2}
    \end{equation}
    
    We claim that $\norm{v_{g(t)}}_{C_\sigma^{2,\alpha}(M)}$ is uniformly bounded, from which the desired result follows.
    Note that $\norm{R_{g(t)}}_{C_{\sigma+2}^{0,\alpha}}$ is uniformly bounded. 
    A function
    \[
        F(g, v) = (-4\Delta_g + R_g)v + R_g \colon C_\eta^{2,\alpha}(S^2T^\ast M) \times C_\sigma^{2,\alpha}(M) \to C_{\sigma+2}^{0,\alpha}(M)
    \]
    is analytic and for a fixed $g \in C_\eta^{2,\alpha}$, $f(g, \cdot)$ is invertible by \Cref{cor:lfi}, so implicit function theorem gives that $v_{g(t)}$ is continuous in $C_\sigma^{2,\alpha}(M)$. In particular, the desired bound follows.
\end{proof}

From \eqref{eqn:fge} and \cite[(3.11)]{DO20}, we obtain a weighted elliptic equation 
\begin{equation}
    \Delta_{f_g}\left( \dfrac{\Tr_g h}{2} - \delta_g[f_g](h) \right) = \dfrac{1}{2} \left( \div_{f_g} (\div_{f_g} h) - \langle \Ric_{f_g}, h \rangle_g \right).
    \label{eqn:fgwe}
\end{equation}

\begin{lemma}
    Let $(M, g$) be an ALF manifold with $m$, $\alpha$, $\eta$, $\sigma$, $g_0$, $g_{\mathrm{RF}}$ as in \Cref{ass:bs}. If $(g(t))_{t \in [0, T]} \subset  B_\epsilon(g_{\mathrm{RF}};C_\sigma^{2,\alpha})$ is a continuous 1-parameter family of metrics on $M$ and $h(t) = \frac{\partial}{\partial t}g(t)$ is bounded in $C_\sigma^{2,\alpha}(M)$ norm, then there exists a uniform constant $C>0$ independent on $t$ such that 
    \[
        \norm{\delta_{g(t)}[f_{g(t)}](h(t))}_{C_\sigma^{2,\alpha}(M)} \le C.
    \] 
    \label{lem:dfgub}
\end{lemma}

\begin{proof}
    We first prove that 
    \[
        \Delta_{f_{g(t)}} := \Delta_{g(t)} - \nabla^{g(t)}_{\nabla^{g(t)} f_{g(t)}}  \colon C_\sigma^{2,\alpha}(M) \to C_{\sigma+2}^{0,\alpha}(M)
    \]
    is invertible. By \Cref{cor:lfi} and since $\nabla^{g(t)}_{\nabla^{g(t)}f_{g(t)}}$ is a compact operator, $\Delta_{f_{g(t)}}$ is Fredholm of index 0. So it is enough to prove that $\Delta_{f_{g(t)}}$ is injective.
    To that end, suppose $\Delta_{f_{g(t)}}v = 0$. Then by integration by parts, 
    \[
        0 = \int_M v(\Delta_{f_{g(t)}}v)e^{-f_{g(t)}} = - \int_M |\nabla^{g(t)}v|^2 e^{-f_{g(t)}},
    \]
    which gives $v \equiv 0$. This concludes that $(\Delta_{f_{g(t)}})_{t \in [0,T]}$ is a family of invertible Fredholm operators.

    Note that \eqref{eqn:fgub1} and \eqref{eqn:fgub2} moreover give that $g \mapsto f_{g}$ is analytic in $C_\eta^{2,\alpha}(M)$.
    Also note that $\Ric_{g(t)}$ and $f_{g(t)}$ are uniformly bounded in $C_{\sigma+2}^{0,\alpha}(M)$ and $C_\sigma^{2,\alpha}(M)$ norm, respectively.
    A function 
    \begin{align*}
        F(g, h, \phi) &= \Delta_{f_g}\left( \dfrac{\Tr_g h}{2} - \phi \right) - \dfrac{1}{2} \left( \div_{f_g}(\div_{f_g} h) - \langle \Ric_{f_g}, h \rangle_g \right)\\
        &\colon C_\eta^{2,\alpha}(S^2T^\ast M) \times C_\sigma^{2,\alpha}(S^2T^\ast M) \times C_\sigma^{2,\alpha}(M) \to C_{\sigma+2}^{0,\alpha}(M)
    \end{align*}
    is then analytic in $g$, $h$, and $\phi$, and for fixed $g$ and small enough $h$ the above function is invertible. By implicit function theorem, $\delta_{g(t)} [f_{g(t)}](h(t))$ is continuous in $C_\sigma^{2,\alpha}(M)$, and in particular the desired bound follows.
\end{proof}

Before we prove the main theorem, we note here that 
\begin{equation}
    \div_{f_g}(\Ric_g + \nabla^{g,2}f_g) = 0.
    \label{eqn:wdf}
\end{equation}
The computation can be found in \cite[(2.11)]{DO20}.

\begin{proof}[Proof of \Cref{thm:mlrf}]
    By \eqref{eqn:vfb} with $h = -2\Ric_{g}$, $f = f_{g}$, and $\phi = \delta_{g}[f_{g}](-2\Ric_{g})$, we obtain
    \begin{align}
        &\int_{B_R^{g_0}} \delta_{g}\left[(|\nabla^{g}f_{g}|_{g}^2 + R_{g})e^{-f_{g}}dV_{g}\right](-2\Ric_{g})\label{eqn:mlrf1}\\
        & = 
        \int_{B_R^{g_0}} 2\langle \Ric_{g}, \Ric_{f_g} \rangle_{g}e^{-f_{g}}
        + \int_{\partial B_R^{g_0}} \langle \nabla^g R_g, n_g \rangle_g e^{-f_g}\nonumber\\
        &\quad + \int_{\partial B_R^{g_0}} -2\langle \Ric_g(\nabla^g f_g) - R_g \nabla^g f_g - \phi \nabla^g f_g, n_g \rangle_g e^{-f_g}. \nonumber
    \end{align}
    Using divergence theorem and \eqref{eqn:wdf}, we also have 
    \begin{align}
        \int_{B_R^{g_0}} \langle \Ric_{f_g}, \nabla^{g,2}f_g \rangle e^{-f_g}
        &= -\int_{B_R^{g_0}} \langle \div_{f_g}(\Ric_{f_g}) , \nabla^g f_g \rangle e^{-f_g}+ \int_{\partial B_R^{g_0}} \langle \Ric_{f_g}(\nabla^g f_g), n_g \rangle e^{-f_g}\label{eqn:mlrf2}\\
        &= \int_{\partial B_R^{g_0}} \langle \Ric_{f_g}(\nabla^g f_g), n_g \rangle e^{-f_g}. \nonumber
    \end{align}
    Adding \eqref{eqn:mlrf1} and \eqref{eqn:mlrf2}, we get 
    \begin{align}
        &\int_{B_R^{g_0}} \delta_{g}\left[(|\nabla^{g}f_{g}|_{g}^2 + R_{g})e^{-f_{g}}dV_{g}\right](-2\Ric_{g})\nonumber\\
        &= \int_{B_R^{g_0}} 2\langle \Ric_{f_g}, \Ric_{f_g} \rangle_g e^{-f_g}+\int_{\partial B_R^{g_0}} \langle \nabla^g R_g, n_g \rangle_g e^{-f_g}\nonumber\\
        &\quad+ \int_{\partial B_R^{g_0}} -2\langle \Ric_g(\nabla^g f_g) - R_g\nabla^g f_g - \phi \nabla^g f_g, n_g \rangle_g e^{-f_g}\label{eqn:mlrf3}\\
        &\quad+ \int_{\partial B_R^{g_0}} -2\langle \Ric_{f_g}(\nabla^g f_g), n_g \rangle_g e^{-f_g} \label{eqn:mlrf4}
    \end{align}
    Note that the boundary terms \eqref{eqn:mlrf3} and \eqref{eqn:mlrf4} uniformly go to zero as $R \to +\infty$ for $g = g(t)$ with $t \in [0,T]$ by \Cref{lem:fgub} and \Cref{lem:dfgub}. Hence for $[t_1, t_2] \subset [0,T]$, by \eqref{eqn:lalfe},
    \begin{align}
        &\lambda_{\ALF}(g(t_2),g_{\mathrm{RF}}) - \lambda_{\ALF}(g(t_1), g_{\mathrm{RF}})\nonumber\\
        &= \lim_{R \to \infty} \left[ \int_{t_1}^{t_2} \int_{B_R^{g_0}} \delta_{g(t)} \left[ (|\nabla^{g(t)}f_{g(t)}|_{g(t)}^2 + R_{g(t)})e^{-f_{g(t)}}dV_{g(t)} \right](-2\Ric_{g(t)})dt \right.\nonumber\\
        &\qquad\qquad \left. - \int_{\partial B_R^{g_0}} \langle \div_{g_{\mathrm{RF}}}(g(t_2)-g(t_1)) - \nabla^{g_{\mathrm{RF}}}\Tr_{g_{\mathrm{RF}}}(g(t_2)-g(t_1)), n_{g_{\mathrm{RF}}} \rangle_{g_{\mathrm{RF}}} dA_{g_{\mathrm{RF}}} \right]\nonumber\\
        &= \lim_{R \to \infty} \left[ \int_{t_1}^{t_2} \left( \int_{B_R^{g_0}} 2\langle \Ric_{g(t)} + \nabla^{g(t),2}f_{g(t)}, \Ric_{g(t)} + \nabla^{g(t),2}f_{g(t)}\rangle_{g(t)}e^{-f_{g(t)}} dV_{g(t)} \right. \right. \nonumber \label{eqn:mlrf5}\\
        &\qquad\qquad + \int_{\partial B_R^{g_0}} \langle \nabla^{g(t)}R_{g(t)}, n_{g(t)}\rangle_{g(t)}e^{-f_{g(t)}}dA_{g(t)}\\
        &\qquad\qquad \left.\left. + 2 \int_{\partial B_R^{g_0}} \langle \div_{g_{\mathrm{RF}}}(\Ric_{g(t)}) - \nabla^{g_{\mathrm{RF}}} R_{g(t)}, n_{g_{\mathrm{RF}}} \rangle_{g_{\mathrm{RF}}} dA_{g_{\mathrm{RF}}} \right) dt \right]. \label{eqn:mlrf6}
    \end{align}
    By the traced Bianchi identity
    \[
        \nabla^{g(t)}R_{g(t)} = 2\div_{g(t)}(\Ric_{g(t)})
    \]
    and since $\norm{g(t)-g_{\mathrm{RF}}}_{C_\sigma^{2,\alpha}(M)}$ and $\norm{\Ric_{g(t)}}_{C_\sigma^{0,\alpha}(M)}$ are uniformly bounded, the terms \eqref{eqn:mlrf5} and \eqref{eqn:mlrf6} cancel, leaving 
    \[
        \lambda_{\ALF}(g(t_2), g_{\mathrm{RF}}) - \lambda_{\ALF}(g(t_1), g_{\mathrm{RF}}) = \int_{t_1}^{t_2} \norm{\Ric_{g(t)} + \nabla^{g(t),2}f_{g(t)}}_{L^2(e^{-f_{g(t)}})}^2dt,
    \]
    or equivalently \eqref{eqn:mlrf}.

    Now suppose that $\lambda_{\ALF}(g(t), g_{\mathrm{RF}})$ is constant. Then $\Ric_{g(t)} + \nabla^{g(t),2}f_{g(t)} = 0$. Tracing this and comparing to \eqref{eqn:fge}, we get 
    \[
        \Delta_g f_g + |\nabla^g f_g|_g^2 = 0.
    \]
    By integration by parts,
    \[
        0 = \lim_{R \to \infty} \int_{\partial B_R^{g_0}} \langle f_g \nabla^g f_g, n_g \rangle_g e^{-f_g}
        = \lim_{R \to \infty} \int_{B_R^{g_0}} |\nabla^g f_g|_g^2 e^{-f_g},
    \]
    which implies that $f_g$ is constant, and therefore $\Ric_{g(t)} \equiv 0$.
\end{proof}

\subsection{First and Second variation of $\lambda_{\ALF}$}

\begin{proposition}[first variation]
    Let $(M, g$) be an ALF manifold with $m$, $\alpha$, $\eta$, $\sigma$, $g_0$, $g_{\mathrm{RF}}$ as in \Cref{ass:bs}. Suppose $g \in B_\epsilon(g_{\mathrm{RF}};C_\sigma^{2,\alpha})$ and $h \in C_\sigma^{2,\alpha}(S^2T^\ast M)$, then 
    \begin{equation}
        \delta_g[\lambda_{\ALF}(g,\overline{g})](h) = -\int_{M} \langle \Ric_{f_g}, h \rangle e^{-f_g}.
        \label{eqn:lfv}
    \end{equation}
\end{proposition}
\begin{proof}
    Using \eqref{eqn:vfb} we can compute the first variation of integrands in \eqref{eqn:lalfe}, and they uniformly converge as $R \to \infty$ by \eqref{eqn:vfu}. Therefore, we get \eqref{eqn:lfv}.
\end{proof}

\begin{proposition}[Second variation]
    Let $(M, g$) be an ALF manifold with $m$, $\alpha$, $\eta$, $\sigma$, $g_0$, $g_{\mathrm{RF}}$ as in \Cref{ass:bs}. Suppose $g \in B_\epsilon(g_{\mathrm{RF}};C_\sigma^{2,\alpha})$ and $h \in C_\sigma^{2,\alpha}(S^2T^\ast M)$, then
    \begin{align}
        \delta_g^2[\lambda_{\ALF}(g, g_{\mathrm{RF}})](h,h)
        &= \dfrac{1}{2} \int_M \langle \Delta_{f_g} h - 2\Rm_g(h) - \cL_{B_{f_g}(h)}(g), h \rangle_g e^{-f_g} \nonumber\\
        &\quad + \dfrac{1}{2} \int_M \langle h \circ \Ric_{f_g} + \Ric_{f_g} \circ h - 2\left( \dfrac{\Tr_g h}{2} - \delta_g[f_g](h) \right) \Ric_{f_g}, h \rangle_g e^{-f_g},
        \label{eqn:lsv}
    \end{align} 
    where $\Rm_g h := R_{ipjq}h^{pq}$\footnote{We use the convention of $R(X,Y)Z = \nabla_Y \nabla_X Z - \nabla_X \nabla_Y Z + \nabla_{[X,Y]}Z$.} and $B_{f_g}(h) := \div_{f_g}h - \nabla^g \left( \frac{\Tr_g h}{2} - \delta_g [f_g](h) \right)$. 
\end{proposition}
\begin{proof}
    The variation formula follows from \cite[(3.8)-(3.10)]{DO20}. Note that the term $\delta_g[f_g](h)$ can be computed explicitly, but vanishes when $h$ is traceless and divergence-free, see \eqref{eqn:fgwe}.
\end{proof}

\section{Dynamical instability and Linear stability of ALF metrics}

Recall that $\lambda_{\ALF}$ is non-decreasing along the Ricci flow. Consequently, if a perturbation of a metric increases the value of $\lambda_{\ALF}$, then the Ricci flow starting from the perturbed metric can never return to the original one. Motivated by this observation, we define the notions of stability and instability for Ricci-flat metrics as follows.

\begin{definition}
    A Ricci-flat ALF metric $g_{\mathrm{RF}}$ is said to be \emph{linearly stable} if the second variation of $\lambda_{\ALF}$ is nonpositive, i.e.
    \[
        \delta_{g_{\mathrm{RF}}}^2[\lambda_{\ALF}](h, h) \le 0 
    \]
    for all symmetric 2-tensors $h$ which satisfy the decay condition $|\nabla^k h| = O(\rho^{-1/2-\epsilon-k})$ for some $\epsilon > 0$. The metric $g_{\mathrm{RF}}$ is said to be \emph{dynamically unstable} otherwise.
\end{definition}

\begin{remark}
    We note that hyperk\"ahler ALF manifolds are linearly stable. This follows from the Weitzenböck formula in \cite[Proposition 11.2]{DO24}. We also note that dynamical stability of Taub-NUT metric among $U(2)$-invariant metrics is shown in \cite{DG21}.
\end{remark}

The linear (in)stability would optimally be computed on $h$ such that $\nabla h\in L^2$. We opted for pointwise controls so that the dynamical instability statement holds thanks to the following result.
\begin{proposition}
    Let $(M, g_{\mathrm{RF}})$ be an ALF Ricci-flat metric and $h \in C_{\sigma}^{2,\alpha}(S^2T^\ast M)$ a divergence-free traceless 2-tensor on $M$ with $\sigma \in (\frac{m-2}{2}, m-2)$ satisfying 
    \begin{equation}
        \delta_{g_\mathrm{RF}}^2[\lambda_{\ALF}](h, h) > 0.
        \label{eqn:insvp}
    \end{equation}
    Then there exists $s>0$ such that $\lambda_{\ALF}(g_{\mathrm{RF}}+sh, g_{\mathrm{RF}}) > 0 = \lambda_{\ALF}(g_{\mathrm{RF}}, g_{\mathrm{RF}})$. In particular, a Ricci flow starting at $g_{\mathrm{RF}}+sh$ will always stay at a definite distance of $g$ in $C_\sigma^{2,\alpha}(S^2T^\ast M)$.
\end{proposition}

\begin{remark}
    Previous results such as \cite{DK24} prove the existence of a perturbation $g+s h$ increasing the scalar curvature. Even though Ricci flow preserves positive scalar curvature and $g$ is Ricci-flat, this is not sufficient to show that Ricci flow stays away from the metric $g$. In fact, on $\mathbb{R}^n$ there are Ricci flows with positive scalar curvature that flow back to the flat metric \cite{Li18}.
\end{remark}

\begin{proof}
    Note that by \Cref{lem:fgub} we may assume that $f_{g_{\mathrm{RF}}+th}$ is uniformly bounded for $0 \le t \le s$ by assuming $s$ is sufficiently small. We will use this fact without mentioning again throughout the proof.

    By the first variation of $\lambda_{\ALF}$, \eqref{eqn:lfv}, we get 
    \begin{align*}
        \lambda_{\ALF}(g_{\mathrm{RF}}+sh, g_{\mathrm{RF}})
        &= \int_0^s \delta_{g_{\mathrm{RF}}+th}[\lambda_{\ALF}](h)dt\\
        &= \int_0^s \left[ \int_M -\langle \Ric_{g_{\mathrm{RF}}+th} + \Hess_{g_{\mathrm{RF}}+th}, h \rangle e^{-f_{g_{\mathrm{RF}}+th}} \right].
    \end{align*}
    We first estimate the Ricci curvature term as follow.
    \begin{align*}
        &\int_M - \langle \Ric_{g_{\mathrm{RF}}+th}, h \rangle e^{-f_{g_{\mathrm{RF}}+th}}\\
        &\quad= t \left( \int_M \frac{1}{2} \langle \Delta_{g_{\mathrm{RF}}}h - 2\Rm_g(h) , h \rangle \right) + t^2 O\left( \int_M \left( |h||\nabla^{g_{\mathrm{RF}},2}h| + |\nabla^{g_{\mathrm{RF}}}h|^2 \right)|h| \right)
    \end{align*}
    For the Hessian term we use integration by parts.
    \begin{align*}
        &\left| \int_M \langle \Hess f_{g_{\mathrm{RF}}+th}, h \rangle e^{-f_{g_{\mathrm{RF}}+th}}\right|\\
        &\quad= \left| - \int_M \langle \nabla^{g_{\mathrm{RF}}+th} f_{g_{\mathrm{RF}}+th}, \div_{f_{g_{\mathrm{RF}}+th}}(h) \rangle e^{-f_{g_{\mathrm{RF}}+th}} \right|\\
        &\quad\le \left| \int_M \langle \nabla^{g_{\mathrm{RF}}+th} f_{g_{\mathrm{RF}}+th}, \div_{g_{\mathrm{RF}}+th}(h) \rangle e^{-f_{g_{\mathrm{RF}}+th}} \right| + \left| \int_M h(\nabla^{g_{\mathrm{RF}}+th} f_{g_{\mathrm{RF}}+th}, \nabla^{g_{\mathrm{RF}}+th} f_{g_{\mathrm{RF}}+th}) e^{-f_{g_{\mathrm{RF}}+th}} \right|\\
        &\quad= O\left( \norm{ e^{-f_{g_{\mathrm{RF}}+th}} \nabla^{g_{\mathrm{RF}}+th} f_{g_{\mathrm{RF}}+th} }_{L^2} \norm{\div_{g_{\mathrm{RF}}+th}(h)}_{L^2} + \norm{h}_{C^0} \norm{ e^{-f_{g_{\mathrm{RF}}+th}} \nabla^{g_{\mathrm{RF}}+th} f_{g_{\mathrm{RF}}+th} }_{L^2}^2 \right)\\
        &\quad= t^2 O\left(\norm{h}_{C^0} \norm{\nabla^{g_{\mathrm{RF}}}h}_{L^2}^2\right).
    \end{align*}
    Here, the last two lines follow from Cauchy-Schwarz inequality and \cite[Proposition 4.1]{DO20}.
    
    Combining these two estimates and the second variation of $\lambda_{\ALF}$, \eqref{eqn:lsv}, we obtain
    \begin{align*}
        &\lambda_{\ALF}(g_{\mathrm{RF}}+sh, g_{\mathrm{RF}})\\
        &\quad= \int_0^s \left[ t\left( \dfrac{1}{2}\Delta_{g_{\mathrm{RF}}} h - \Rm_g(h) \right) + t^2 O\left( \norm{h}_{C^0} \left( \norm{\rho^{-1}h}_{L^2} \norm{\rho \nabla^{g_{\mathrm{RF}},2}h}_{L^2} + \norm{\nabla^{g_{\mathrm{RF}}}h}_{L^2}^2 \right) \right)  \right]dt\\
        &\quad=\frac{1}{2}s^2 \left( \frac{1}{2} \Delta_{g_{\mathrm{RF}}} h - \Rm_g(h) \right) + s^3 O\left( \norm{h}_{C^0} \left( \norm{\rho^{-1}h}_{L^2} \norm{\rho \nabla^{g_{\mathrm{RF}},2}h}_{L^2} + \norm{\nabla^{g_{\mathrm{RF}}}h}_{L^2}^2 \right) \right)\\
        &\quad= \frac{1}{2}s^2 \delta_{g_{\mathrm{RF}}}^2 [\lambda_{\ALF}](h,h) + s^3 O\left( \norm{h}_{C^0} \left( \norm{\rho^{-1}h}_{L^2} \norm{\rho \nabla^{g_{\mathrm{RF}},2}h}_{L^2} + \norm{\nabla^{g_{\mathrm{RF}}}h}_{L^2}^2 \right) \right).
    \end{align*}
    Therefore, for small enough $s>0$ we have $\lambda_{\ALF}(g_{\mathrm{RF}}+sh, g_{\mathrm{RF}}) > 0$ given \eqref{eqn:insvp}. The computation also shows that $\lambda_{\ALF}$ is continuous in $C_\sigma^{2,\alpha}(S^2 T^\ast M)$. Since $\lambda_{\ALF}$ is non-decreasing by \Cref{thm:mlrf}, the desired result follows.
\end{proof}

Our notion of (in)stability agrees with the definitions given in \cite[Definition 12]{BO23} for traceless, divergence-free deformations $h$ with suitable decay, since for such $h$ we have 
\[
    \delta_{g_{\mathrm{RF}}}^2[\lambda_{\ALF}](h, h) = \delta_{g_{\mathrm{RF}}}^2[\mathcal{S}](h, h),
\]
where $\mathcal{S} \colon g \mapsto \int R_g$ denotes the Einstein-Hilbert functional.

In \cite{BO23}, it is proved that the Kerr, Taub-Bolt, and Chen-Teo metrics are unstable by constructing explicit variations $h$ for which $\delta_g^2[\mathcal{S}](h, h) < 0$. However, the explicitly constructed $h$ is not divergence-free, and making it so required the use of an analytic tool, called $b$-calculus.
We show that $h$ can be made divergence-free using the analysis developed in \Cref{sec:fl}, thus avoiding the reliance on the heavy machinery of $b$-calculus. This yields a more self-contained and accessible proof of instability for the Kerr, Taub-Bolt, and Chen-Teo metrics.

\begin{proof}[Proof of Theorem \ref{thm:dyn instability}]
    We recall the construction of the variation $h$ from \cite{BO23}. Let $g_{\mathrm{RF}}$ denote one of the Kerr, Taub-Bolt, or Chen-Teo metrics. Let $\lambda$ be the simple (nonzero) eigenvalue of $W_{g_{\mathrm{RF}}}^+$, the self-dual part of the Weyl curvature, and define $\hat{f} := \lambda^{-1/3}=O(r)$, up to a multiplicative constant. It is known that the conformally rescaled metric $g_K := \hat{f}^2 g_{\mathrm{RF}}$ is an extremal K\"ahler metric. In particular, the vector field 
\[
    X := J \nabla^{g_{\mathrm{RF}}}\hat{f} = O(1)
\]
is Killing. Let $K := g_{\mathrm{RF}}(X, \cdot)=O(1)$ be the 1-form dual to $X$, and define the (anti-)self-dual 2-forms $\omega_\pm := d_\pm K=O(r^{-2})$. Then the symmetric 2-tensor $h:=\hat{f}^3 \omega_- \circ \omega_+=O(r^{-1})$ is the variation considered in \cite{BO23}.

In \Cref{sec:rfdf}, we show that for each of the Kerr, Taub-Bolt, and Chen-Teo metrics, there exists a function $\psi \in C^{2,\alpha}(M)$ such that, denoting $\div_0^\ast$ the traceless part of $\div^\ast_{g_{\mathrm{RF}}}$ 
\[
    \div_{g_{\mathrm{RF}}}(h - \div_0^\ast d\psi) = 0
\]
and $\div_0^\ast d\psi \in C_{3/4}^{1,\alpha}(M)$.
Moreover, \cite[Proposition 15]{BO23} shows that $h$ has the required decay.

Define the divergence-free perturbation $h_0 := h - \div_0^\ast d\psi$. By construction, we will have $h_0=O(r^{-2})$ since the leading term of $h$ and $\div_0^\ast d\psi$ coincide. By \cite{BO23}, $\div_g(\hat{f}^{-3}h) = 0$ and $\Delta_L h = c\hat{f}^{-3}h=O(r^{-4})$, where $\Delta_L = \Delta_g - 2\Rm_g$ denotes the Lichnerowicz Laplacian, for a positive constant $c$ depending on the multiplicative factor of $\hat{f}$ with respect to $\lambda^{-1/3}$. Then the second variation of $\lambda_{\ALF}$ in the direction $h_0$ becomes 
\[
    \delta_{g_{\mathrm{RF}}}^2[\lambda_{\ALF}](h_0, h_0) = \frac{1}{2}\int_M \langle \Delta_L h_0, h_0 \rangle = \frac{1}{2} \int_M \langle \Delta_L h_0, h \rangle = c\int_M \hat{f}^{-3}|h|^2 > 0,
\]
which establishes the instability.
\end{proof}

\section{A Positive Mass Theorem}
Motivated by \cite{BO22}, we define the weighted relative mass with a weight $f$ on an ALF manifold $M$ to be 
\[
    m_f(g, \overline{g}) := m(g, \overline{g}) + 2 \lim_{R \to \infty} \int_{\partial B_R^{g_0}} \langle \nabla^{g_0}f, n \rangle_{g_0}e^{-f}.
\]
The weight we choose is $f_g$ in \eqref{eqn:fg}. Consider the conformal metric $\widetilde{g} := e^{-2f_g/m}g$ as in \cite{LLS25}. Since $f_g \in C_\sigma^{2,\alpha}(M)$, the conformal metric $\widetilde{g}$ also satisfies 
\[
    \widetilde{g} - g_0 \in C_\eta^{2,\alpha}(S^2T^\ast M), \quad \widetilde{g} - g_{\mathrm{RF}} \in C_\sigma^{2,\alpha}(S^2T^\ast M).
\]
Moreover, the scalar curvature of conformal metric is 
\[
    R_{\widetilde{g}} = e^{\frac{2f_g}{m}} \left( R_g + 2\Delta_g f_g 
    - \frac{m-1}{m} |\nabla^g f_g|_g^2 \right) = \frac{1}{m}e^{\frac{2f_g}{m}} |\nabla^g f_g|^2,
\]
so in particular it is integrable. Therefore, the weighted relative mass is well-defined.
\begin{proposition}
    Let $g \in B_\epsilon(g_{\mathrm{RF}}; C_\sigma^{2,\alpha})$ with integrable scalar curvature and $f \in C_\sigma^{2,\alpha}(M)$. Consider the conformal metric $\widetilde{g} = e^{-\frac{2f}{m}}g$. Then the weighted relative mass of $(M, g, f)$ equals the relative mass of $(M, \widetilde{g})$, i.e.
    \[
        m_{f}(g, g_{\mathrm{RF}}) = m(\widetilde{g}, g_{\mathrm{RF}}).
    \]
    \label{prop:wmm}
\end{proposition}

\begin{proof}
    Note that since $f \in C_\sigma^{2,\alpha}(M)$, $g-g_{\mathrm{RF}} \in C_\sigma^{2,\alpha}(S^2T^\ast M)$, and $2\sigma + 1 > m-1$, all subleading terms in the integrand vanish in the limit.
    \begin{align*}
    m(\widetilde{g}, g_{\mathrm{RF}})
    &= \lim_{R \to \infty} \int_{\partial B_R^{g_0}} \langle \div_{g_{\mathrm{RF}}} (\widetilde{g} - g_{\mathrm{RF}}) - \nabla^{g_{\mathrm{RF}}} \Tr_{g_{\mathrm{RF}}} (\widetilde{g} - g_{\mathrm{RF}}), n_{g_{\mathrm{RF}}} \rangle_{g_{\mathrm{RF}}}\\
    &= \lim_{R \to \infty} \int_{\partial B_R^{g_0}} e^{-\frac{2f}{m}} \left[ m(g) - \dfrac{2}{m} \langle \nabla^{g_{\mathrm{RF}}} f - (\Tr_{g_{\mathrm{RF}}}g) \nabla^{g_{\mathrm{RF}}} f, n_{g_{\mathrm{RF}}} \rangle_{g_{\mathrm{RF}}} \right]\\
    &= m(g) + 2 \int_{\partial B_R^{g_0}} \langle \nabla^{g_{\mathrm{RF}}}f_g, n_{g_{\mathrm{RF}}} \rangle_{g_{\mathrm{RF}}} e^{-f}\\
    &= m_f(g, g_{\mathrm{RF}}). \qedhere
    \end{align*}
\end{proof}

\begin{proposition}
    Let $g \in B_\epsilon(g_{\mathrm{RF}}; C_\sigma^{2,\alpha})$ with integrable scalar curvature and $f_g$ be as in \eqref{eqn:fg}. Then 
    \[
        m_{f_g}(g, g_{\mathrm{RF}}) = -\lambda_{\ALF}(g, g_{\mathrm{RF}}).
    \]
    \label{prop: wml}
\end{proposition}
\begin{proof}
    Note that by \Cref{prop:lcalf}, $\lambda_{\ALF}^\circ$ is attained by the unique solution $w_g$ satisfying \eqref{eqn:lcalf2}. Together with integration by parts then yields the following.
    \begin{align*}
        \lambda_{\ALF}(g, g_{\mathrm{RF}})
        &= \inf_{w - 1 \in C_c^\infty(M)} \int_M (4 |\nabla^g w|_g^2 + R_g w^2) - m(g, g_{\mathrm{RF}})\\
        &= \int_M (4|\nabla^g w_g|_g^2 + R_gw_g^2) - m(g, g_{\mathrm{RF}})\\
        &= \lim_{R \to \infty} \left[ \int_{B_R^{g_0}} w_g(-4\Delta_g w_g + R_g w_g) + \int_{\partial B_R^{g_0}} 4w_g \langle \nabla^g w_g, n \rangle \right] - m(g, g_{\mathrm{RF}})\\
        &= \lim_{R \to \infty} \int_{\partial B_R^{g_0}} -2\langle \nabla^g f_g, n \rangle e^{-f_g} - m(g, g_{\mathrm{RF}})\\
        &= -m_{f_g}(g, g_{\mathrm{RF}}). \qedhere
    \end{align*}
\end{proof}

These two propositions give that whenever we have a positive (relative) mass theorem for metrics with nonnegative scalar curvature, $\lambda_{\ALF}(g, \overline{g})$ has a sign up to the suitable choice of reference metric $\overline{g}$. In particular, if $(M, g)$ is spin and AF, then \cite{Dai04} shows that the positive mass theorem indeed holds. 

\begin{corollary}
    For a spin AF manifold $(M,g)$ with nonnegative scalar curvature, one has
\[
    \lambda_{\ALF}(g, \overline{g}) \le 0
\]
where $\overline{g}$ is asymptotic to $g_{\mathbb{R}^m \times S^1}$ at infinity. There is equality if and only if $g$ and $\overline{g}$ are isometric.

This inequality also provides a lower bound for mass:
\begin{equation}
    0\leq \int_M (4|\nabla^g w_g|_g^2 + R_gw_g^2) \leq m(g, \overline{g}).
\end{equation}
\end{corollary}

\begin{remark}
    The above negativity of $\lambda_{\ALF}$ is interpreted as a stability result. The spin assumption is crucial: for instance, the inequality does not hold for Kerr metrics (and in particular Schwarzschild metric).
\end{remark}

We finally prove a positive mass theorem for spin manifolds asymptotic to the flat $(\mathbb{R}^3\times\mathbb{S}^1)/\mathbb{Z}_2$ with compatible spin structure. The equality case is that of the hyperkähler metrics of Hitchin--Page \cite{Pag81, Hit84}. By \cite{BM11}, these hyperkähler metrics are ALF, asymptotic to the flat metric $g_0$ to order $r^{-3+\varepsilon}$ for any $\varepsilon>0$, hence have zero mass relatively to $g_0$. By virtue of being hyperkähler, these metrics also carry a parallel spinor $\psi_0$ that is asymptotic to a parallel spinor of the flat metric $g_0$. 

\begin{proof}[Proof of Theorem \ref{thm:HP}]
    By taking the cover $\mathbb{R}^3\times\mathbb{S}^1\to(\mathbb{R}^3\times\mathbb{S}^1)/\mathbb{Z}_2$, we can apply the analysis of \cite{Dai04} and the associated extension of Witten's formula for mass to manifolds asymptotic to $(\mathbb{R}^3\times\mathbb{S}^1)/\mathbb{Z}_2$. Namely, one finds up to multiplying $\psi_0$ constant 
    \begin{equation}\label{eq:mass spin}
        m(g,g_0)=\int_M\big(4|\nabla\psi|^2 + \operatorname{R}|\psi|^2\big)dv_g,
    \end{equation}
    where $\psi$ is the harmonic spinor asymptotic to $\psi_0$. In the case of equality, one finds that $\psi$ is parallel, hence $g$ is hyperkähler. Note that Kato's inequality and \eqref{eq:mass spin} also tell us that if $\operatorname{R}\geq 0$ and if $\operatorname{R}\in L^1$, then $m(g,g_0)\geq \widetilde{\mathcal{F}}(|\psi|,g)\geq \lambda^\circ_{\ALF}(g)\geq 0$. 

    Combining the above positive mass theorem and the equalities from \Cref{prop:wmm} and \Cref{prop: wml}, we deduce that if $g$ is asymptotic to $(\mathbb{R}^3\times\mathbb{S}^1)/\mathbb{Z}_2$ at a rate $\frac12<\tau<1$ and has nonnegative scalar curvature, then $\lambda_{\ALF}(g,g_0)\leq 0$ with equality if and only if $g$ is hyperkähler.
\end{proof}

\appendix
\section{Gauge-fixing of Perturbations on Conformally K\"ahler ALF Metrics}

In this appendix, we explain how to modify the variations constructed in \cite{BO23} to make them divergence-free. The strategy is the same for all three metrics. We begin by observing that the divergence of the given variation can be expressed as the differential of a function depending on the radial coordinate $r$ and a fiber coordinate at infinity. This function can be written as the Laplacian of a well-behaved function that captures the correct asymptotic behavior. By subtracting this, we obtain that the remaining term can be written as the Laplacian of another function with suitable decay. This leads to a correction term of $h$.

Although the initial correction functions may appear singular near the origin, they should be interpreted as having the specific asymptotic behavior only outside a compact set. Within the compact region, the function can be arbitrarily defined. The resulting error introduced in this region is controlled by the Fredholm theory developed earlier, and thus does not affect the validity of the construction.

\subsection{Euclidean Kerr metric} \label{sec:rfdf}
The Kerr metric is defined for parameters $m > 0$ and $a$ as
\[
    g_{\mathrm{Kerr}} = \rho^2 \left( \dfrac{1}{\Psi}dr^2 + d\theta^2 \right) + \dfrac{\Psi}{\rho^2}(d\tau + a\sin^2\theta d\phi)^2 + \dfrac{\sin^2\theta}{\rho^2}(ad\tau - (r^2-a^2)d\phi)^2,
\]
where $\rho^2 = r^2 - a^2\cos^2\theta$ and $\Psi = r^2-2mr-a^2$. 

For the Kerr metric, $\lambda = \dfrac{4m}{(r-a\cos\theta)^3}$, so we take $\hat{f} = r-a\cos\theta$. The Killing vector field and the associated 1-form are 
\begin{align*}
    J \nabla^{g_{\mathrm{Kerr}}} \hat{f} &= \partial_\tau,\\
    K &= \dfrac{r^2-2mr-a^2\cos^2\theta}{r^2-a^2\cos^2\theta}d\tau - \dfrac{2mra\sin^2\theta}{r^2-a^2\cos^2\theta}d\phi.
\end{align*}
Then for $\omega_- = d_-K$,
\[
    \div_{g_{\mathrm{Kerr}}}h = -3\omega_-(K) = \dfrac{-3m(dr - a\sin\theta d\theta)}{2(r+a\cos\theta)^2} = d\left( \dfrac{3m}{2(r+a \cos\theta)} \right).
\]

On the Kerr metric, we have 
\begin{align*}
    \Delta_{g_{\mathrm{Kerr}}} r &= - \dfrac{2(r-m)}{r^2-a^2\cos^2\theta} = \dfrac{2}{r} - \dfrac{2m}{r^2} + O(r^{-3}),\\
    \Delta_{g_{\mathrm{Kerr}}} \log r &= \dfrac{r^2+a^2}{r^2(r^2-a^2\cos^2\theta)} = \dfrac{1}{r^2} + O(r^{-3}),\\
    \Delta_{g_{\mathrm{Kerr}}} \cos\theta &= -\dfrac{2\cos\theta}{r^2-a^2\cos^2\theta} = -\dfrac{2\cos\theta}{r^2} + O(r^{-3}).
\end{align*}
Therefore, 
\[
    \div_{g_{\mathrm{Kerr}}} h = \dfrac{3}{4} \Delta_{g_{\mathrm{Kerr}}} \left( mr + 2m^2 \log r + am \cos\theta + \varphi \right)
\]
for some $\varphi \in C_{3/4}^{2,\alpha}(M)$ by \Cref{cor:lfi}.
Then we get 
\[
    \div_{g_{\mathrm{Kerr}}}\left(h - \div_0^\ast \left( \left(m + \frac{2m^2}{r}\right)dr - am\sin\theta d\theta + d\varphi \right)\right) = 0
\]
Note that in the case of the Schwarzschild metric with $a=0$, we find $\phi=0$.

\subsection{Taub-Bolt metric}
The Taub-Bolt metric is defined for parameters $m = \frac{5}{4}n > 0$ as 
\[
    g_{\mathrm{TB}} = \rho^2\left( \dfrac{1}{\Psi}dr^2 + d\theta^2 + \sin^2\theta d\phi^2 \right) + 4n^2 \dfrac{\Psi}{\rho^2} (d\psi + \cos \theta d\phi)^2,
\]
where $\rho^2 = r^2-n^2$ and $\Psi = r^2 - 2mr + n^2$.

For the Taub-Bolt metric, $\lambda = \dfrac{9n}{(r+n)^3}$, so we take $\hat{f} = r+n$. The Killing vector field and the associated 1-form are 
\begin{align*}
    J\nabla^{g_{\mathrm{TB}}}\hat{f} &= 2n\partial_\psi,\\
    K &= 2n \dfrac{r^2-2mr+n^2}{r^2-n^2}d\psi.
\end{align*}
Then for $\omega_- = d_-K$, 
\[
    \div_{g_{\mathrm{TB}}}h = -3\omega_-(K) = -\dfrac{3n}{8(r-n)^2} = d \left( \dfrac{3n}{8(r-n)} \right).
\]

On the Taub-Bolt metric, we have 
\begin{align*}
    \Delta_{g_{\mathrm{TB}}}r &= \dfrac{4r-5n}{2r^2-2n^2} = \dfrac{2}{r} - \dfrac{5n}{2r^2} + O(r^{-3}),\\
    \Delta_{g_{\mathrm{TB}}} \log r &= \dfrac{1}{r^2}.
\end{align*}

Therefore, 
\[
    \div_{g_{\mathrm{TB}}} h = \dfrac{3}{4} \Delta_{g_{\mathrm{TB}}}\left( \dfrac{n}{4}r + \dfrac{9n}{8}\log r + \varphi \right)
\]
for some $\varphi \in C_{3/4}^{2,\alpha}(M)$ by \Cref{cor:lfi}. Then we get 
\[
\div_{g_{\mathrm{TB}}}\left(h - \div_0^\ast \left( \left( \frac{n}{4} + \frac{9n}{8r} \right)dr + d\varphi \right)\right) = 0.
\]

\subsection{Chen-Teo metric}
The Chen-Teo metric is defined as 
\[
    g_{\mathrm{CT}} = \dfrac{\kappa H(x,y)}{(x-y)^3} \left( \dfrac{dx^2}{X} - \dfrac{dy^2}{Y} - \dfrac{XY}{\kappa F(x,y)}d\phi^2 \right) + \dfrac{F(x,y)}{(x-y)H(x,y)}\left( d\tau + \dfrac{G(x,y)}{F(x,y)}d\phi \right)^2,
\]
where 
\[
    X = P(x) \quad \text{and} \quad Y = P(y)
\]
for a quartic polynomial $P(t) = a_0 + a_1t + a_2t^2 + a_3t^3 + a_4t^4$ and $F(x,y)$, $H(x,y)$ $G(x,y)$ are polynomials in $x$ and $y$ given by 
\begin{align*}
    F(x,y) &= y^2 X - x^2 Y\\
    H(x,y) &= (\nu x + y) [(\nu x - y)(a_1 - a_3xy) - 2(1-\nu)(a_0 - a_4x^2y^2)]\\
    G(x,y) &= (\nu^2 a_0 + 2\nu a_3 y^3 + 2\nu a_4 y^4 - a_4 y^4)X + (a_0 - 2\nu a_0 - 2\nu a_1 x - \nu^2 a_4 x^4)Y
\end{align*} 
for parameters $\kappa$ and $\nu \in (-1,1)$. 

For the Chen-Teo metric, $\lambda = -\dfrac{2(\nu+1)}{\kappa} \left( \dfrac{x-y}{\nu x+y} \right)^3$, so we take $\hat{f} = \dfrac{\sqrt{\kappa}}{\nu+1} \dfrac{\nu x + y}{x-y}$. The Killing vector field is 
\begin{align*}
    J\nabla^{g_{\mathrm{CT}}}\hat{f} = \partial_\tau.
\end{align*}
Then by \cite[Remark 3.2 and (4.13)]{BK21}, for $\omega_- = d_-K$ as before,
\[
\div_{g_{\mathrm{CT}}}h = -3\omega_-(K) = d \left( \dfrac{-3(x-y)(\nu x+y)(a_1 - a_3 xy)}{(\nu-1)H} \right).
\]

As in \cite[Section 4.2.1]{BK21}, we further assume that $P(x)$ is a cubic polynomial with roots 
\[
    x_1 = -4\xi^3 (1-\xi), \quad
    x_2 = -\xi(1-2\xi+2\xi^2), \quad
    x_3 = 1-2\xi,
\]
and a formal fourth root $x_4 = \infty$, and take $\nu = -2\xi^2$. 
This give the two-parameter family of Chen-Teo metrics in $(\kappa, \xi)$. Reparametrize $x_1 < y < x_2 < x < x_3$ as 
\[
    x = x_2 - \dfrac{x_2 \sqrt{\kappa (1-\nu^2)}}{r}\cos^2\dfrac{\theta}{2}, \quad y = x_2 + \dfrac{x_2 \sqrt{\kappa (1-\nu^2)}}{r} \sin^2\dfrac{\theta}{2}.
\]
Under this reparametrization, we have 
\[
    \div_{g_{\mathrm{CT}}}h = \dfrac{c_1}{r} + \dfrac{c_2 + c_3\cos\theta}{r^2} + O(r^{-3}),
\]
where $c_1, c_2, c_3$ are constants depending only on $\xi$.

Also the Laplacian of several functions are as follows:
\begin{align*}
    \Delta_{g_{\mathrm{CT}}} r &= \dfrac{2}{r} - \dfrac{2\sqrt{\kappa}(1+2\xi^2)^{3/2}}{(1-2\xi^2)^{1/2}} \dfrac{1}{r^2} + O(r^3),\\
    \Delta_{g_{\mathrm{CT}}} \log r &= \dfrac{1}{r^2} + O(r^{-3}),\\
    \Delta_{g_{\mathrm{CT}}} \cos\theta &= -\dfrac{2\cos\theta}{r^2} + O(r^{-3}).
\end{align*}

Therefore, for constants $c_1'$, $c_2'$, $c_3'$ depending on $\xi$, 
\[
    \div_{g_{\mathrm{CT}}} h = \dfrac{3}{4}\Delta_{g_{\mathrm{CT}}} (c_1' r + c_2' \log r + c_3' \cos\theta + \varphi)
\]
for some $\varphi \in C_{3/4}^{2,\alpha}(M)$  by \Cref{cor:lfi}. 
Then we get 
\[
\div_{g_{\mathrm{CT}}}\left(h - \div_0^\ast \left( \left( c_1' + \frac{c_2'}{r} \right)dr -c_3'\sin\theta d\theta +  d\varphi\right)\right) = 0.
\]

\bibliographystyle{amsalpha}
\bibliography{refs}

\providecommand{\bysame}{\leavevmode\hbox to3em{\hrulefill}\thinspace}
\providecommand{\MR}{\relax\ifhmode\unskip\space\fi MR }
\providecommand{\MRhref}[2]{%
  \href{http://www.ams.org/mathscinet-getitem?mr=#1}{#2}
}
\providecommand{\href}[2]{#2}
\begin{thebibliography}{HHM04}

\bibitem[AA24]{AA24}
Steffen Aksteiner and Lars Andersson, \emph{Gravitational instantons and special geometry}, J. Differential Geom. \textbf{128} (2024), no.~3, 928--958. \MR{4809333}

\bibitem[AA25]{AA25}
Lars Andersson and Bernardo Araneda, \emph{Charges, complex structures, and perturbations of instantons}, arXiv preprint arXiv:2503.24265v2 (2025).

\bibitem[Bar86]{Bar86}
Robert Bartnik, \emph{The mass of an asymptotically flat manifold}, Comm. Pure Appl. Math. \textbf{39} (1986), no.~5, 661--693. \MR{849427}

\bibitem[BG23]{BG23}
Olivier Biquard and Paul Gauduchon, \emph{On toric {H}ermitian {ALF} gravitational instantons}, Comm. Math. Phys. \textbf{399} (2023), no.~1, 389--422. \MR{4567377}

\bibitem[BGL24]{BGL24}
Olivier Biquard, Paul Gauduchon, and Claude LeBrun, \emph{Gravitational instantons, {W}eyl curvature, and conformally {K}\"ahler geometry}, Int. Math. Res. Not. IMRN (2024), no.~20, 13295--13311. \MR{4811689}

\bibitem[BK21]{BK21}
Thomas~John Baird and Hari Kunduri, \emph{Abelian instantons over the {C}hen-{T}eo {AF} geometry}, J. Geom. Phys. \textbf{168} (2021), Paper No. 104310, 23. \MR{4278133}

\bibitem[BM11]{BM11}
Olivier Biquard and Vincent Minerbe, \emph{A {K}ummer construction for gravitational instantons}, Comm. Math. Phys. \textbf{308} (2011), no.~3, 773--794. \MR{2855540}

\bibitem[BO23a]{BO22}
Julius Baldauf and Tristan Ozuch, \emph{The spinorial energy for asymptotically {E}uclidean {R}icci flow}, Adv. Nonlinear Stud. \textbf{23} (2023), no.~1, Paper No. 20220045, 22. \MR{4576078}

\bibitem[BO23b]{BO23}
Olivier Biquard and Tristan Ozuch, \emph{Instability of conformally k$\backslash$" ahler, einstein metrics}, arXiv preprint arXiv:2310.10109 (2023).

\bibitem[BW12]{BW12}
T.~Balehowsky and E.~Woolgar, \emph{The {R}icci flow of asymptotically hyperbolic mass and applications}, J. Math. Phys. \textbf{53} (2012), no.~7, 072501, 15. \MR{2985224}

\bibitem[BW18]{BW18}
Eric Bahuaud and Eric Woolgar, \emph{Asymptotically hyperbolic normalized {R}icci flow and rotational symmetry}, Comm. Anal. Geom. \textbf{26} (2018), no.~5, 1009--1045. \MR{3900478}

\bibitem[CC19]{CC19}
Gao Chen and Xiuxiong Chen, \emph{Gravitational instantons with faster than quadratic curvature decay ({II})}, J. Reine Angew. Math. \textbf{756} (2019), 259--284. \MR{4026454}

\bibitem[CC21a]{CC21a}
\bysame, \emph{Gravitational instantons with faster than quadratic curvature decay. {I}}, Acta Math. \textbf{227} (2021), no.~2, 263--307. \MR{4366415}

\bibitem[CC21b]{CC21b}
\bysame, \emph{Gravitational instantons with faster than quadratic curvature decay ({III})}, Math. Ann. \textbf{380} (2021), no.~1-2, 687--717. \MR{4263695}

\bibitem[CLN06]{CLN06}
Bennett Chow, Peng Lu, and Lei Ni, \emph{Hamilton's {R}icci flow}, Graduate Studies in Mathematics, vol.~77, American Mathematical Society, Providence, RI; Science Press Beijing, New York, 2006. \MR{2274812}

\bibitem[CT11]{CT11}
Yu~Chen and Edward Teo, \emph{A new {AF} gravitational instanton}, Phys. Lett. B \textbf{703} (2011), no.~3, 359--362. \MR{2831866}

\bibitem[CVZ23]{CVZ23}
Gao Chen, Jeff Viaclovsky, and Ruobing Zhang, \emph{Hodge theory on {$\rm ALG^*$} manifolds}, J. Reine Angew. Math. \textbf{799} (2023), 189--227. \MR{4595310}

\bibitem[Dai04]{Dai04}
Xianzhe Dai, \emph{A positive mass theorem for spaces with asymptotic {SUSY} compactification}, Comm. Math. Phys. \textbf{244} (2004), no.~2, 335--345. \MR{2031034}

\bibitem[DG21]{DG21}
Francesco Di~Giovanni, \emph{Convergence of {R}icci flow solutions to {T}aub-{NUT}}, Comm. Partial Differential Equations \textbf{46} (2021), no.~8, 1521--1568. \MR{4286466}

\bibitem[DK24]{DK24}
Mattias Dahl and Klaus Kr\"oncke, \emph{Local and global scalar curvature rigidity of {E}instein manifolds}, Math. Ann. \textbf{388} (2024), no.~1, 453--510. \MR{4693938}

\bibitem[DO20]{DO20}
Alix Deruelle and Tristan Ozuch, \emph{A lojasiewicz inequality for ale metrics}, arXiv preprint arXiv:2007.09937 (2020).

\bibitem[DO24]{DO24}
\bysame, \emph{Orbifold singularity formation along ancient and immortal ricci flows}, arXiv preprint arXiv:2410.16075 (2024).

\bibitem[Gib80]{Gib80}
G.~W. Gibbons, \emph{Gravitational instantons: a survey}, Mathematical problems in theoretical physics ({P}roc. {I}nternat. {C}onf. {M}ath. {P}hys., {L}ausanne, 1979), Lecture Notes in Phys., vol. 116, Springer, Berlin-New York, 1980, pp.~282--287. \MR{582633}

\bibitem[GP80]{GP80}
G.~W. Gibbons and Malcolm~J. Perry, \emph{New gravitational instantons and their interactions}, Phys. Rev. D (3) \textbf{22} (1980), no.~2, 313--321. \MR{578954}

\bibitem[GT77]{GT}
David Gilbarg and Neil~S. Trudinger, \emph{Elliptic partial differential equations of second order}, Grundlehren der Mathematischen Wissenschaften, vol. Vol. 224, Springer-Verlag, Berlin-New York, 1977. \MR{473443}

\bibitem[Has11]{Has11}
Robert Haslhofer, \emph{A renormalized {P}erelman-functional and a lower bound for the {ADM}-mass}, J. Geom. Phys. \textbf{61} (2011), no.~11, 2162--2167. \MR{2827116}

\bibitem[Has12]{Has12}
\bysame, \emph{Perelman's lambda-functional and the stability of {R}icci-flat metrics}, Calc. Var. Partial Differential Equations \textbf{45} (2012), no.~3-4, 481--504. \MR{2984143}

\bibitem[Hei11]{Hei11}
Hans-Joachim Hein, \emph{Weighted {S}obolev inequalities under lower {R}icci curvature bounds}, Proc. Amer. Math. Soc. \textbf{139} (2011), no.~8, 2943--2955. \MR{2801635}

\bibitem[Hei12]{Hei12}
\bysame, \emph{Gravitational instantons from rational elliptic surfaces}, J. Amer. Math. Soc. \textbf{25} (2012), no.~2, 355--393. \MR{2869021}

\bibitem[HHM04]{HHM04}
Tam\'as Hausel, Eugenie Hunsicker, and Rafe Mazzeo, \emph{Hodge cohomology of gravitational instantons}, Duke Math. J. \textbf{122} (2004), no.~3, 485--548. \MR{2057017}

\bibitem[Hit84]{Hit84}
N.~J. Hitchin, \emph{Twistor construction of {E}instein metrics}, Global {R}iemannian geometry ({D}urham, 1983), Ellis Horwood Ser. Math. Appl., Horwood, Chichester, 1984, pp.~115--125. \MR{757213}

\bibitem[HM14]{HM14}
Robert Haslhofer and Reto M\"uller, \emph{Dynamical stability and instability of {R}icci-flat metrics}, Math. Ann. \textbf{360} (2014), no.~1-2, 547--553. \MR{3263173}

\bibitem[HSW07]{HSW07}
Gustav Holzegel, Thomas Schmelzer, and Claude Warnick, \emph{Ricci flows connecting {T}aub-bolt and {T}aub-{NUT} metrics}, Classical Quantum Gravity \textbf{24} (2007), no.~24, 6201--6217. \MR{2374541}

\bibitem[Hug24]{Hug24}
John Hughes, \emph{${L}^2$-instability of the taub-bolt metric under the ricci flow}, arXiv preprint arXiv:2408.15269 (2024).

\bibitem[Hug25a]{Hug25b}
\bysame, \emph{An ancient ricci flow emerging from taub-bolt}, arXiv preprint arXiv:2509.23276 (2025).

\bibitem[Hug25b]{Hug25a}
\bysame, \emph{Towards a taub-bolt to taub-nut via ricci flow with surgery}, arXiv preprint arXiv:2509.23257 (2025).

\bibitem[HW06]{HW06}
Matthew Headrick and Toby Wiseman, \emph{Ricci flow and black holes}, Classical Quantum Gravity \textbf{23} (2006), no.~23, 6683--6707. \MR{2273515}

\bibitem[KW25]{KW25}
Marcus Khuri and Jian Wang, \emph{Mass lower bounds for asymptotically locally flat manifolds}, arXiv preprint arXiv:2509.03014v1 (2025).

\bibitem[Li18]{Li18}
Yu~Li, \emph{Ricci flow on asymptotically {E}uclidean manifolds}, Geom. Topol. \textbf{22} (2018), no.~3, 1837--1891. \MR{3780446}

\bibitem[Li23]{Li23}
Mingyang Li, \emph{Classification results for conformally kähler gravitational instantons}, arXiv preprint arXiv:2310.13197v4 (2023).

\bibitem[LLS25]{LLS25}
Michael~B. Law, Isaac~M. Lopez, and Daniel Santiago, \emph{Positive mass and {D}irac operators on weighted manifolds and smooth metric measure spaces}, J. Geom. Phys. \textbf{209} (2025), Paper No. 105386, 20. \MR{4833746}

\bibitem[LS25]{LS25}
Mingyang Li and Song Sun, \emph{Gravitational instantons and harmonic maps}, arXiv preprint arXiv:2507.15284 (2025).

\bibitem[Min09]{Min09}
Vincent Minerbe, \emph{A mass for {ALF} manifolds}, Comm. Math. Phys. \textbf{289} (2009), no.~3, 925--955. \MR{2511656}

\bibitem[MS12]{MS12}
Donovan McFeron and G\'abor Sz\'ekelyhidi, \emph{On the positive mass theorem for manifolds with corners}, Comm. Math. Phys. \textbf{313} (2012), no.~2, 425--443. \MR{2942956}

\bibitem[OW07]{OW07}
Todd~A. Oliynyk and Eric Woolgar, \emph{Rotationally symmetric {R}icci flow on asymptotically flat manifolds}, Comm. Anal. Geom. \textbf{15} (2007), no.~3, 535--568. \MR{2379804}

\bibitem[Pag81]{Pag81}
Don~N. Page, \emph{A periodic but nonstationary gravitational instanton}, Phys. Lett. B \textbf{100} (1981), no.~4, 313--315. \MR{609994}

\bibitem[Per02]{Per02}
Grisha Perelman, \emph{The entropy formula for the ricci flow and its geometric applications}, arXiv preprint math/0211159 (2002).

\bibitem[Ses06]{Ses06}
Natasa Sesum, \emph{Linear and dynamical stability of {R}icci-flat metrics}, Duke Math. J. \textbf{133} (2006), no.~1, 1--26. \MR{2219268}

\bibitem[Shi89]{Shi89}
Wan-Xiong Shi, \emph{Deforming the metric on complete {R}iemannian manifolds}, J. Differential Geom. \textbf{30} (1989), no.~1, 223--301. \MR{1001277}

\bibitem[SZ24]{SZ24}
Song Sun and Ruobing Zhang, \emph{Collapsing geometry of hyperk\"ahler 4-manifolds and applications}, Acta Math. \textbf{232} (2024), no.~2, 325--424. \MR{4816632}

\bibitem[Tak14]{Tak14}
Ryosuke Takahashi, \emph{An ancient solution of the {R}icci flow in dimension 4 converging to the {E}uclidean {S}chwarzschild metric}, Comm. Anal. Geom. \textbf{22} (2014), no.~2, 289--342. \MR{3210757}

\end{thebibliography}
\end{document}